\documentclass[11pt, leqno]{amsart} 

\usepackage[foot]{amsaddr}
\usepackage[mathscr]{eucal}
\usepackage{amsmath}
\usepackage{amsfonts}
\usepackage{amsthm}
\usepackage{color,xcolor}
\usepackage{xypic}
\usepackage{comment}

\theoremstyle{plain}
\newtheorem{theorem}{Theorem}[section]
\newtheorem{proposition}{Proposition}[section]
\newtheorem{lemma}{Lemma}[section]
\newtheorem{remark}{Remark}

\usepackage{a4wide}
\usepackage{amssymb}
\usepackage{graphicx}

\newtheorem{example}{Example}[section]

\numberwithin{equation}{section}

\def\Vec#1{\mbox{\boldmath $#1$}}

\dedicatory{Dedicated to professor Tatsuru Takakura on the occasion of his 60th birthday}

\title[Homogeneity of magnetic trajectory]
{Homogeneity of magnetic trajectories in the Berger sphere}

\author[J.~Inoguchi]{Jun-ichi Inoguchi}
\address[J. I.]
{Department of Mathematics,
Hokkaido University
Sapporo
060-0810 Japan}
\email{inoguchi@math.sci.hokudai.ac.jp}

\author[M.~I.~Munteanu]{Marian Ioan Munteanu}
\address[M.~I.~M.: corresponding author]
{University 'Al. I. Cuza' of Iasi, 
Faculty of Mathematics, Bd. Carol I, no.~11,
700506 Iasi, Romania}
\email{marian.ioan.munteanu@gmail.com}

\date{\today}

\begin{document}

\begin{abstract}
We study the homogeneity of contact magnetic trajectories in naturally reductive
Berger spheres. We prove that every contact magnetic trajectory is a product of 
a homogeneous geodesic and a charged Reeb flow.


\end{abstract}

\maketitle
\section{Introduction}
Normal homogeneous spaces of positive curvature are classified by 
Berger \cite{Berger} and Wilking \cite{Wil}. 
In the list given by Berger \cite{Berger}, one can see a 
normal homogeneous space of the form 
$(\mathrm{SU}(2)\times\mathbb{R})/H_{r}$ 
diffeomorphic to the 
unit $3$-sphere $\mathbb{S}^3$ (see Appendix \ref{sec:A} for the definition of the isotropy group $H_r$). 
The normal homogeneous 
space $(\mathrm{SU}(2)\times\mathbb{R})/H_{r}$ is nowaday called 
the \emph{Berger sphere}. Berger sphere appears in many 
branches of differential geometry. 
For instance Berger sphere is a simple example 
which exhibits collapsing of Riemannian manifolds. 
Under a certain limit in Gromov-Hausdorff sense, 
Berger sphere collapses to the $2$-sphere. 

The Berger sphere admits a 
homogeneous contact structure compatible with the metric. 
The Berger sphere equipped with a homogeneous 
contact structure is (homothetic to) a Sasakian manifold of constant 
holomorphic sectional curvature greater than $1$. 
In this article we generalize the notion of Berger sphere as a 
complete and simply connected Sasakian $3$-manifold 
$\mathscr{M}^3(c)$ of constant holomorphic sectional curvature 
$c>-3$ ($c\not=1$). The homogeneous Riemannian space $\mathscr{M}^3(c)$ with $-3<c<1$ is no longer 
normal homogeneous but still naturally reductive.

In this decade, surface geometry of Berger sphere becomes an active 
area of submanifold geometry. Montaldo and Onnis studied 
helix surfaces in the Berger sphere \cite{MO}. Here a \emph{helix surface} means a surface whose 
normal direction makes constant angle 
with the Reeb vector field of the contact form 
of $\mathscr{M}^3(c)$. The Berger sphere 
is represented by $\mathscr{M}^3(c)=\mathrm{U}(2)/\mathrm{U}(1)$ 
as a naturally reductive homogeneous space. 
A surface $\varSigma$ is $\mathscr{M}^3(c)$ is said to be
a \emph{rotationally invariant surface} if it is invariant under 
the $\mathrm{U}(1)$-action. 
Rotationally invariant surfaces of constant mean curvature 
are studied in Torralbo \cite{Torralbo10}. 
Torralbo and Van der Veken \cite{ToVe} studied 
rotationally invariant surfaces of constant Gauss curvature. 
Torralbo investigated compact minimal surfaces \cite{Torralbo12}. 

In this paper, we address Hamiltonian systems on the Berger sphere $\mathscr{M}^3(c)$. 
We start our investigation with geodesics in $\mathscr{M}^3(c)$. 
After the publication of the seminal paper \cite{Berger} by Berger, 
global and local differential geometric properties of the Berger sphere have been well studied
in detail. For instance, Jacobi fields along geodesics, cut loci are investigated in 
\cite{Ra,Sakai81,Z}.
In the present work we re-examine geodesic from 
Hamiltonian dynamics viewpoint. 
We interpret the equation of geodesic 
on a Riemannian manifold $(M,g)$ as a 
Hamiltonian system on 
the cotangent bundle (\emph{phase space}) 
$T^{*}M$ with respect to the 
canonical symplectic structure of $T^{*}M$ and 
focus on its integrability. The Hamiltonian 
derived from the geodesic equation is the \emph{kinetic energy}. The solutions
of the Hamiltonian system derived from the 
geodesic equation are called \emph{geodesic flows} \cite{BT,HMR}. 
One of the fundamental problems of Hamiltonian systems 
is to investigate the \emph{preservation of integrability} 
or preservation of the existence of periodic orbits 
under perturbations of the symplectic structure. 

A \emph{magnetic field} on a Riemannian manifold $(M,g)$ 
is formulated mathematically as a closed $2$-form on $M$. 
By virtue 
of a  magnetic field $F$ on a base manifold $M$, 
one obtains a perturbed symplectic form on the cotangent bundle $T^{*}M$. 
The solutions to the perturbed Hamiltonian system of geodesic flows are
called \emph{magnetic geodesic flows} \cite{AnSi,Arnold61,Arnold86}. 
The curves on the configuration space $M$ 
obtained as projection images of magnetic geodesic flows 
are called \emph{magnetic trajectories} 
(see Appendix \ref{sec:C}). 
As nice perturbations 
of geodesic flows, magnetic trajectories have been paid much attention of 
differential geometers as well as researchers of 
dynamical systems. See \textit{e.g.}, 
\cite{BJ,IM221,MSY} and references therein.

Let us consider geodesic flows of homogeneous 
Riemannian spaces. Because of the homogeneity, 
we 
may concentrate our attention to geodesics starting at the origin. 
It is known that every geodesic of a naturally reductive homogeneous 
Riemannian space is homogeneous. 
More precisely every geodesic starting at 
the origin of a naturally reductive homogeneous space 
is an orbit of a one-parameter subgroup of the (largest) 
connected isometry group (see \textit{e.g.} \cite{Av,DZ,KV}). 
This fact implies that the Hamiltonian system of 
geodesic flows of a 
naturally reductive homogeneous Riemannian space 
is completely integrable (in non-commutative sense). 

In \cite{BJ}, 
Bolsinov and Jovanovi{\'c} 
studied magnetic trajectories in \emph{normal} homogeneous Riemannian spaces. 
They showed that every magnetic trajectory starting at the origin under the influence of the \emph{standard invariant magnetic field} (see \eqref{eq:homF}) is 
homogeneous (see \cite[Remark 1]{BJ}). Thus 
magnetic trajectories in naturally reductive homogeneous Riemannian spaces would be 
next targets. 

Now, let us return our attention to the Berger sphere 
$\mathscr{M}^{3}(c)$. 
As we mentioned above, $\mathscr{M}^3(c)$ is naturally reductive 
for any $c$. Thus every geodesic is homogeneous. 
The Riemannian metric of the Berger sphere is 
obtained from the standard Riemannian 
metric of the unit $3$-sphere by the perturbing 
the fiber components with respect to the 
Hopf fibering. The connection form of the 
Hopf fibering is nothing but the standard contact form 
of $\mathbb{S}^3$. 
The curvature form of the connection form gives a 
homogeneous magnetic field on the Berger sphere. 
We call it the \emph{contact magnetic field} of 
a Berger sphere $\mathscr{M}^{3}(c)$. 
From this construction we can expect that 
magnetic trajectories with respect to the contact magnetic 
field of $\mathscr{M}^{3}(c)$ are strongly affected by 
both Riemannian and contact geometric properties of the Berger sphere. 
Obviously the notion of geodesic is a Riemannian notion, 
\textit{i.e.} only depends on Riemannian structure. 
On the other hand the notion of contact magnetic trajectory 
depends on both Riemannian structure and 
contact structure.

Motivated by these observations, in our previous work 
\cite{IM17}, the present authors studied periodicity of contact magnetic 
trajectories in $\mathscr{M}^3(c)$. In addition, 
the Jacobi fields for contact magnetic trajectories are completely 
determined in \cite{IM222}. 
It should be noted that 
Bimmermann and Maier computed the Hofer-Zender capacity 
of certain lens spaces \cite{BM}. In the work \cite{BM}, 
magnetic trajectores on the lens space $L(p;1)=\mathbb{S}^3/\mathbb{Z}_p$ induced by the contact magnetic fields on the $3$-sphere $\mathbb{S}^3$ play a crucial role. Moreover contact magnetic fields are used to 
construct certain spacetimes in general relativity \cite{I-M,KIKM}.

Having in mind the 
homogeneity results of magnetic trajectories 
in normal homogeneous spaces due to 
Bolsinov and Jovanovi{\'c} \cite{BJ}, we study 
the homogeneity of contact magnetic trajectories 
in naturally reductive Berger spheres. 

The present paper is organized as follows. 
First we give explicit parametrization 
of geodesics (starting at the origin) 
of the Berger sphere 
$\mathscr{M}^3(c)$ in homogeneous geodesic form in 
Theorem~\ref{THM3.1}. To this end, we give naturally 
reductive homogeneous space representation of 
$\mathscr{M}^3(c)$ explicitly in Section~\ref{section1}. 
In Section~\ref{sec4} we prove that every contact magnetic trajectory  
of $\mathscr{M}^3(c)$ is homogeneous. 
In particular we prove that every 
contact magnetic trajectory is a product of 
homogeneous geodesic and a (charged) Reeb flow.

As a result, the integrability and homogeneity 
of geodesic flows of the Berger sphere is preserved under 
the perturbation by the contact magnetic field. 

Before closing Introduction, we mention our previous work \cite{DRIMN15}.
The Berger sphere $\mathscr{M}^3(c)$ together with its 
metric and contact structure is generalized to 
arbitrary odd-dimension in a straightforward manner. 
The resulting $(2n+1)$-dimensional Berger sphere $\mathscr{M}^{2n+1}(c)$ 
is naturally reductive for any $c$ and normal homogeneous 
for $c\geq 1$. In \cite{DRIMN15}, 
we proved a \emph{codimension reduction theorem} 
for contact magnetic trajectories 
in $\mathscr{M}^{2n+1}(c)$. 
More precisely, it is proved that the essential dimension 
for the theory of contact magnetic trajectories in 
$\mathscr{M}^{2n+1}(c)$ is $3$. 
Recently, 
Albers, Benedetti and Maier \cite{ABM} gave an 
alternative and new proof for the codimension 
reduction theorem for contact magnetic trajectories 
in the unit sphere $\mathbb{S}^{2n+1}$.

\section{Preliminaries}
\subsection{Homogeneous geometry}\label{sec:1.1}
Let $M=L/H$ be a homogeneous manifold. 
Every element $a\in L$ acts transitively 
on $M$. 
For a point $p=aH$ of $M=L/H$, 
the image 
of $p$ under $b\in L$ is denoted 
by $b\cdot p$. By definition, $b\cdot p=(ba)H$.

Let us denote by $\mathfrak{l}$ and $\mathfrak{h}$ the Lie algebras 
of $L$ and $H$, respectively. 
Then $M$ is said to be \emph{reductive} if there exists a 
linear subspace $\mathfrak{p}$ (called a \emph{Lie subspace}) of 
$\mathfrak{l}$ complementary 
to $\mathfrak{h}$ and satisfies 
$[\mathfrak{h},\mathfrak{p}]\subset\mathfrak{p}$. It is known that 
every homogeneous Riemannian space is reductive. 

Now let $M=L/H$ be a homogeneous Riemannian space with 
reductive decomposition $\mathfrak{l}=\mathfrak{h}+\mathfrak{p}$ 
and an $L$-invariant Riemannian metric 
$g=\langle\cdot,\cdot\rangle$. Then $M$ is said to be \emph{naturally reductive} with respect to $\mathfrak{p}$ 
if the $\mathsf{U}_\mathfrak{p}$ tensor vanishes, where
$\mathsf{U}_\mathfrak{p}$ is defined by
\begin{equation}
\label{eq:NatRed}
2\langle \mathsf{U}_\mathfrak{p}(X,Y),Z\rangle=
\langle X, [Z,Y]_{\mathfrak{p}}\rangle
+\langle Y, [Z,X]_{\mathfrak{p}}\rangle.
\end{equation}
for any $X$, $Y$, $Z\in\mathfrak{p}$. Here 
we denote the $\mathfrak{p}$ part of a vector 
$X\in\mathfrak{l}$ by $X_{\mathfrak{p}}$. 
A homogeneous Riemannian space $(M,g)$ is 
said to be \emph{naturally reductive} if 
it admits an isometry group $L$ acting 
transitively on $M$ and admitting a reductive decomposition 
$\mathfrak{l}=\mathfrak{h}+\mathfrak{p}$ which is 
naturally reductive. 

On a naturally reductive homogeneous space 
$M=L/H$ with naturally reductive 
decomposition $\mathfrak{l}=\mathfrak{h}+\mathfrak{p}$, 
the tangent space $T_{o}M$ of $M$ at the origin $o=H$ is 
identified with $\mathfrak{p}$.

\begin{proposition}\label{prop:NR-geodesic}
Every geodesic $\gamma(s)$ of a naturally 
reductive homogeneous space $L/H$ with 
naturally reductive decomposition $\mathfrak{l}=
\mathfrak{h}+\mathfrak{p}$ starting at the origin $o$ with 
initial velocity $X\in \mathfrak{p}$ is \emph{homogeneous}, that is, 
it is expressed as
\[
\gamma(s)=\exp_{L}(sX)\cdot o,
\]
where $\exp_{L}:\mathfrak{l}\to L$ is the 
exponential map.
\end{proposition}

A homogeneous Riemannian space $M=L/H$ is said to be 
\emph{normal} if $L$ is compact semi-simple and 
the metric is derived from the restriction 
of a bi-invariant Riemannian metric of $L$. 
For instance, let $\mathsf{B}$ be the Killing form 
of the Lie algebra $\mathfrak{l}$. 
Then, for any positive constant $\alpha$,  
$\langle\cdot,\cdot\rangle_{\mathsf B}=-\alpha\,\mathsf{B}$ 
induces a bi-invariant Riemannian metric on $L/H$. 
Then we obtain a naturally reductive 
decomposition $\mathfrak{l}=\mathfrak{h}\oplus\mathfrak{p}$, where 
$\mathfrak{p}=\mathfrak{h}^{\perp}$ is the orthogonal complement of 
$\mathfrak{h}$ in $\mathfrak{l}$ with respect to 
$\langle\cdot,\cdot\rangle_{\mathsf B}$. Thus normal homogeneous spaces are naturally reductive. 
Riemannian symmetric spaces are also naturally reductive.

As a generalization of naturally reductive homogeneous 
space, the notion of Riemannian g.~o.~space 
was introduced by Kowalski and Vanhecke \cite{KV}. 

According to \cite{KV}, a homogeneous Riemannian space 
$M$ is called a 
\emph{space with homogeneous geodesics} or a 
\emph{Riemannian g.o.~space} if every geodesic 
$\gamma(s)$ of $M$ is an orbit of a one-parameter subgroup of 
the \emph{largest} connected group of isometries. 
Proposition \ref{prop:NR-geodesic} implies that naturally reductive homogenous spaces are Riemannian g.o.~spaces. 
As pointed out by Arnold (\cite{Arnold66}, \cite[p.~379]{Arnold}), homogeneous geodesics 
describe relative equilibria of a mechanical system, such as rotations
with constant angular velocity.
 For more information on 
Riemannian g.o.~spaces, we refer to \cite{Av}. 

\subsection{The Euler-Arnold equation}\label{sec:1.2}
Let $G$ be a Lie group equipped with
a left invariant Riemannian metric $\langle\cdot,\cdot\rangle$. 
The bi-invariance obstruction ${\mathsf{U}}$ is a 
symmetric bilinear map 
$\mathsf{U}:\mathfrak{g}\times\mathfrak{g}\to\mathfrak{g}$ 
defined by
\begin{equation}\label{eq:U}
2\langle {\mathsf{U}}(X,Y),Z \rangle=
-\langle X,[Y,Z]\rangle+\langle Y,[Z,X] \rangle,
\quad  X,Y \in \mathfrak{g}
\end{equation}
The Levi-Civita connection $\nabla$ is described as
\[
\nabla_{X}Y=\frac{1}{2}[X,Y]+\mathsf{U}(X,Y),
\quad X,Y\in\mathfrak{g}.
\]
On the Lie algebra $\mathfrak{g}$, we define 
the linear operator $\mathrm{ad}^{*}:\mathfrak{g}
\to\mathfrak{gl}(\mathfrak{g})$ by
\[
\langle \mathrm{ad}(X)Y,Z\rangle
=\langle Y,\mathrm{ad}^{*}(X)Z\rangle,
\quad X,Y,Z\in\mathfrak{g}.
\]
One can see that 
\[
-2\mathsf{U}(X,Y)=\mathrm{ad}^{*}(X)Y+\mathrm{ad}^{*}(Y)X.
\]
Take a curve $\gamma(s)$ starting at the origin (identity $\mathbf{1}$) of $G$. 
Set $\varOmega(s)=\gamma(s)^{-1}\dot{\gamma}(s)$, then 
one can check that 
\[
\nabla_{\dot{\gamma}}\dot{\gamma}
=\gamma\left(
\dot{\varOmega}-\mathrm{ad}^{*}(\varOmega)\varOmega
\right).
\]
This implies the 
following fundamental fact \cite{Arnold}:
\begin{proposition}
The curve $\gamma(s)$ is a geodesic if and only if $\varOmega$ satisfies
\[
\dot{\varOmega}=\mathrm{ad}^{*}(\varOmega)\varOmega.
\]
\end{proposition}
Now let us assume that $G$ is \emph{compact} and semi-simple. Then 
$G$ admits bi-invariant Riemannian metrics. As we saw in Section \ref{sec:1.1}, we obtain a bi-invariant Riemannian metric on $G$ 
induced from the inner product 
$\langle\cdot,\cdot\rangle_{\mathsf B}=-\alpha\,\mathsf{B}$. 
Here $\mathsf{B}=\mathsf{B}_{\mathfrak g}$ is the 
Killing form of the Lie algebra $\mathfrak{g}$ of $G$. 
The positive constant $\alpha$ will 
be appropriately chosen to normalize the range of sectional curvature in later sections.

Next, introduce a left invariant endomorphism field 
$\mathcal{I}$ (called the 
\emph{moment of inertia tensor field}) on $\mathfrak{g}$ by
\begin{equation}\label{eq:Inertia}
\langle X,Y\rangle=\langle\mathcal{I}X,Y\rangle_{\mathsf B},
\quad \mathcal{I}X,Y\in\mathfrak{g}.
\end{equation}
Since $\mathrm{ad}(X)$ is skew-adjoint with respect to $\langle\cdot,\cdot\rangle_1$ for any 
$X\in\mathfrak{g}$, we obtain
\[
\langle \mathrm{ad}^{*}(\varOmega)\varOmega,Z\rangle
=\langle \mathcal{I}^{-1}([\mathcal{I}\varOmega,\varOmega]),Z\rangle.
\]
Henceforth we obtain
\[
\langle\nabla_{\dot{\gamma}}\dot{\gamma},\gamma Z
\rangle
=\langle \dot{\varOmega}-\mathcal{I}^{-1}[\mathcal{I}\varOmega,\varOmega],Z
\rangle, \quad Z\in\mathfrak{g}.
\]
Let us introduce the 
\emph{momentum} $\mu$ by
\[
\mu=\mathcal{I}\varOmega.
\]
Then $\dot{\varOmega}=\mathcal{I}^{-1}\dot{\mu}$ and hence 
\[
\langle\nabla_{\dot{\gamma}}\dot{\gamma},\gamma Z
\rangle
=\langle \mathcal{I}^{-1}(\dot{\mu}-[\mu,\varOmega]),Z
\rangle, \quad Z\in\mathfrak{g}.
\]
Thus we arrive at the so-called 
\emph{Euler-Arnold equation} \cite{Arnold} 
(also called the \emph{Euler-Poincar{\'e} equation} \cite{BoMa,HMR})
\[
\dot{\mu}-[\mu,\varOmega]=0.
\]

\begin{example}[Rigid bodies]\label{eg:RigidBody}{\rm Let us consider the 
rotation group $G=\mathrm{SO}(3)$. 
The Lie algebra $\mathfrak{g}=\mathfrak{so}(3)$
is parametrized as
\[
\left\{
X=\left(
\left.
\begin{array}{ccc}
0 & -x_3 & x_2\\
x_3 & 0 & -x_1\\
-x_2 & x_1 & 0
\end{array}
\right)
\>\right\vert
\>x_1,x_2,x_3\in\mathbb{R}
\right\}.
\]
The Killing form $\mathsf{B}$ of $\mathfrak{so}(3)$ is given by
\[
\mathsf{B}(X,Y)=\mathrm{tr}(XY),\quad X,Y\in\mathfrak{so}(3).
\]
Let us 
equip the bi-invariant metric on $\mathrm{SO}(3)$ by
$\langle\cdot,\cdot\rangle_{\mathsf{B}}=-\frac{1}{2}\mathsf{B}$. 
Then the correspondence 
\begin{equation}\label{eq:so(3)R3}
\Vec{x}=\left(
\begin{array}{c}
x_1\\
x_2\\
x_3
\end{array}
\right)
\longleftrightarrow 
X=\left(
\begin{array}{ccc}
0 & -x_3 & x_2\\
x_3 & 0 & -x_1\\
-x_2 & x_1 & 0
\end{array}
\right)
\end{equation}
gives an isometric Lie algebra isomorphims 
from $(\mathbb{R}^3,\times)$ and $\mathfrak{so}(3)$. 
Here $\times$ denotes the cross product of $\mathbb{R}^3$.

The motion of a rigid body in Euclidean $3$-space 
$\mathbb{E}^3=(\mathbb{R}^3,\langle\cdot,
\cdot\rangle)$ is 
described as a curve $\gamma(t)$ in $\mathrm{SO}(3)$. 
Then,
$\varOmega_{\mathsf{c}}:=\gamma(t)^{-1}\dot{\gamma}(t)$ and 
$\varOmega_{\mathsf{s}}:=\dot{\gamma}(t)\gamma(t)^{-1}$ are 
\emph{angular velocity in the body} (\emph{corpus}) and 
\emph{angular velocity in the space}, respectively \cite{AM,Arnold}. 
The motion of a rigid body in 
Euclidean $3$-space $\mathbb{E}^3$ under inertia is a geodesic 
in $\mathrm{SO}(3)$ equipped with a left invariant metric 
$\langle\cdot,\cdot\rangle$. The left 
invariant metric is determined by the 
rigid body's mass distribution. 
The mass distribution is described by an 
endomorphism $\mathcal{I}$ on 
$\mathfrak{so}(3)$ which is 
self-adjoint with respect to the 
inner product $\langle\cdot,\cdot\rangle_{\mathsf B}$.
The endomorphism $\mathcal{I}$ is regarded as a left invariant 
endomorphism field on $\mathrm{SO}(3)$ and called 
the (moment of) inertia tensor field. 
The left invariant Riemannian metric 
$\langle\cdot,\cdot\rangle$ is expressed by 
the formula \eqref{eq:Inertia} (see \textit{e.g.}, \cite[\S 4.2]{AM}). 
Since the moment of inertia tensor field $\mathcal{I}$ is a 
self-adjoint 
linear operator on $\mathfrak{so}(3)$ with respect to $\langle\cdot,\cdot\rangle_{\mathsf{B}}$, it is diagonalizable. 
Let $\{I_1,I_2,I_3\}$ be the eigenvalues 
with orthonormal eigenbasis $\{e_1,e_2,e_3\}$ of $\mathfrak{so}(3)$ such that 
$\mathcal{I}e_i=I_{i}e_i$ ($i=1,2,3$) \cite[\S 5.4]{GPS}.  
The eigenvalues $I_1,I_2,I_3$ 
are called the \emph{principal moments of inertia}. 
The axes $\mathbb{R}e_1$, $\mathbb{R}e_2$ and 
$\mathbb{R}e_3$ are called the \emph{principal axes}.

Without loss of generality, we may assume that 
\[
\Vec{e}_1=\left(
\begin{array}{ccc}
0 & 0 & 0\\
0 & 0 & -1\\
0 & 1 & 0
\end{array}
\right),
\quad 
\Vec{e}_2=
\left(
\begin{array}{ccc}
0 & 0 & 1\\
0 & 0 & 0\\
-1 & 0 & 0
\end{array}
\right),
\quad 
\Vec{e}_3=\left(
\begin{array}{ccc}
0 & -1 & 0\\
1 & 0 & 0\\
0 & 0 & 0
\end{array}
\right).
\]
Let us identify $\mathfrak{so}(3)$ with $(\mathbb{R}^3,\times)$ via 
\eqref{eq:so(3)R3}. Then $\mathcal{I}$ is identified with
\[
\left(
\begin{array}{ccc}
I_1 & 0 & 0\\
0 & I_2 & 0\\
0 & 0 & I_3
\end{array}
\right).
\]
Moreover $\varOmega_{\mathsf c}(t)$ is identified with
$\Vec{\omega}(t)=(\omega_{1}(t),\omega_{2}(t),\omega_{3}(t))$
and hence $\mu(t)$ is identified with the vector valued function 
$\Vec{\mu}(t)=(I_1\omega_{1}(t),I_2\omega_{2}(t),I_3\omega_{3}(t))$. 
Thus the Euler-Arnold equation is rewritten as
\[
\dot{\Vec{\mu}}(t)=\Vec{\mu}(t)\times\Vec{\omega}(t).
\]
Note that the Euler-Arnold equation is 
rewritten as
\begin{equation}\label{eq:EA-SO3}
\mathcal{I}\dot{\Vec{\omega}}(t)=\mathcal{I}\Vec{\omega}(t)
\times\Vec{\omega}(t).
\end{equation}
More explicitly, we have the ODE system
\[
I_{1}\dot{\omega_1}(t)=(I_2-I_3)\omega_{2}(t)\omega_{3}(t),
\quad 
I_{2}\dot{\omega_2}(t)=(I_3-I_1)\omega_{3}(t)\omega_{1}(t),
\quad
I_{3}\dot{\omega_3}(t)=(I_1-I_2)\omega_{1}(t)\omega_{2}(t).
\]
When $I_1=I_2=I_3=k\not=0$, then  we have $\omega_1,\omega_2$ and 
$\omega_3$ are constant. Hence $\Vec{\omega}=\Vec{\omega}(0)=:\Vec{\omega}_0$. Thus we obtain the solution $\gamma(t)$ to the Euler-Arnold equation through the identity matrix satisfying the 
initial condition $\Vec{\omega}(0)=\Vec{\omega}_0$ 
as the (homogeneous) geodesic
\[
\gamma(t)=\exp_{\mathfrak{so}(3)}(t\Vec{\omega}_0)
\]
in $\mathrm{SO}(3)$ with respect to 
the bi-invariant metric $\langle\cdot,\cdot\rangle_{\mathsf B}$ through the identity matrix. 

Next, we observe the case $I_1=I_2\not=I_3$. 
Let us use the eigenspace decomposition
\[
\mathfrak{so}(3)=\mathfrak{m}_{1}\oplus\mathfrak{m}_{3},
\]
where
\[
\mathfrak{m}_{1}=\mathbb{R}e_1\oplus\mathbb{R}e_2,
\quad 
\mathfrak{m}_{3}=\mathbb{R}e_3.
\]
Decompose the initial angular velocity 
$\varOmega_{\mathsf c}(0)$ as 
$\varOmega_{\mathsf c}(0)=\varOmega_{\mathsf{c},1}+\varOmega_{\mathsf{c},3}$ along the eigenspace decomposition. 
Then 
the solution $\gamma(t)$ is given by \cite[Example 1.2]{Souris}:
\begin{equation}\label{eq:symmetricEA}
\gamma(t)=
\exp_{\mathfrak{so}(3)}\left\{
t
\left(
\varOmega_{\mathsf{c},1}+\frac{I_3}{I_1}\varOmega_{\mathsf{c},3}
\right)
\right\}
\exp_{\mathfrak{so}(3)}\left\{
t
\left(
\frac{I_1-I_3}{I_1}\varOmega_{\mathsf{c},3}
\right)
\right\}.
\end{equation}
This kind of Euler-Arnold equation appears in 
the so-called symmetric Euler top as well as the precession 
of the Earth's rotational axis 
(\cite[p.~2610]{Souris}, see also \cite[\S 5.6]{GPS}). 
Note that Souris \cite{Souris} used the inner product $2\langle\cdot,\cdot\rangle$ 
on $\mathfrak{so}(3)$.}
\end{example}
\begin{remark}
{\rm 
In Example \ref{eg:RigidBody}, we only discussed 
\emph{torque free} rotational 
motions of rigid bodies
about a fixed point or the center of mass. 
The Euler-Arnold equation of a 
rotational motion of a rigid body with torque 
$\Vec{N}(t)$ has the form:
\[
\mathcal{I}\dot{\Vec{\omega}}(t)=\mathcal{I}\Vec{\omega}(t)\times\Vec{\omega}(t)+\Vec{N}(t).
\]
Typical examples 
of rotational motions of rigid bodies with torque and symmetry 
$I_1=I_2\not=I_3$ are so called 
\emph{heavy symmetric top} with one point fixed. 
For detailed descriptions, see \cite[\S 5.7]{GPS}. 
}
\end{remark}

\subsection{Static magnetism}
Let $(M,g)$ be a Riemannian manifold. 
A (static) \emph{magnetic field} is a closed 
$2$-form $F$ on $M$. The closeness of $F$ is interpreted as 
the \emph{Gauss' law} of the magnetic field $F$. 
The \emph{Lorentz force} $\varphi$ derived from 
$F$ is an endomorphim field defined by
\[
F(X,Y)=g(\varphi X,Y),
\quad X,Y\in\varGamma(TM).
\]
A curve $\gamma(s)$ is said to be a 
\emph{magnetic trajectory} under the influence 
of $F$ if it obeys the 
\emph{Lorentz equation}
\begin{equation}\label{eq:Lorentz}
\nabla_{\dot{\gamma}}\dot{\gamma}=q\varphi\dot{\gamma},
\end{equation}
where $\nabla$ is the Levi-Civita connection of $(M,g)$, 
$\dot{\gamma}$ is the velocity of $\gamma(s)$ and 
$q$ is a constant called the \emph{charge}.

The Lorentz equation implies that 
magnetic trajectories are of constant speed.
Note that when $F=0$ or $q=0$, magnetic trajectories 
reduce to geodesics. For more information on 
magnetic trajectories, we refer to \cite{IM221}. 
We exhibit two classes of static magnetic fields in 
Section \ref{sec:1.4} and \ref{sec:1.5}.

\subsection{Contact magnetic field}\label{sec:1.4}
A $1$-form $\eta$ on a $3$-manifold $M$ is said to be 
a \emph{contact form} if it satisfies 
$\mathrm{d}\eta\wedge \eta\not=0$ on $M$.
A $3$-manifold $M$ together with a 
contact form $\eta$ is called a 
\emph{contact $3$-manifold}. 
A contact $3$-manifold $(M,\eta)$ is orientable. 
We can take a volume element 
$\mathrm{d}v_{\eta}=-\eta\wedge \mathrm{d}\eta/2$. 
Moreover, there exists a 
unique vector field $\xi$ satisfying 
\[
\eta(\xi)=1,\quad \mathrm{d}\eta(\xi,\cdot)=0.
\]
The vector field $\xi$ is called the 
\emph{Reeb vector field} of $(M,\eta)$. 
The plane field
\[
\mathcal{D}=\{X\in TM\>|\>\eta(X)=0\}
\]
is called the \emph{contact structure} of $(M,\eta)$. 

On contact $3$-manifolds, there are two particular classes 
of curves. The integral curves of $\xi$ are called 
\emph{Reeb flows}. On the other hand, curves tangent to 
$\mathcal{D}$ are 
called \emph{Legendre curves}. Both 
Reeb flows and Legendre curves play 
important roles in contact topology (see \textit{e.g}., \cite{Geiges}). 

We may develop static magnetism on 
contact $3$-manifolds. We regard the 
contact form $\eta$ as a \emph{magnetic potential}. 
The magnetic field $F=\mathrm{d}\eta$ is called 
the \emph{contact magnetic field}. There are infinitely many 
Riemannian metrics on $(M,\eta)$. 
The Lorentz force $\varphi$ of $F$ is 
skew-adjoint with respect to any Riemannian metric $g$ on 
$M$. We may require appropriate \emph{compatibility} of $g$ with respect to 
the Lorentz force $\varphi$. Concerning on compatibility, 
we recall the following fundamental fact (see \textit{e.g.} \cite{CFG,Ch-Ha,MPS}).
\begin{proposition}
On a contact $3$-manifold $(M,\eta)$, there exists a 
Riemannian metric $g$ satisfying 
\[
g(\varphi X,\varphi Y)=g(X,Y)-\eta(X)\eta(Y),
\quad 
g(\varphi X,Y)=\mathrm{d}\eta(X,Y),
\quad g(X,\xi)=\eta(X)
\]
for all vector fields $X$ and $Y$ on $M$. 
The metric $g$ is called the 
compatible metric of $\eta$. The volume element 
$\mathrm{d}v_g$ of $g$ coincides with $\mathrm{d}v_{\eta}$. 
The Lorentz force $\varphi$ of $F=\mathrm{d}\eta$ 
satisfies 
\[
\varphi^2=-\mathrm{I}+\eta\otimes\xi.
\]
\end{proposition}
With respect to the compatible metric $g$, the Reeb 
vector field $\xi$ is a unit vector field.

\begin{remark}{\rm
Note that compatible metric $g$ is 
\emph{not} uniquely determined for 
prescribed contact form $\eta$. 
The naturally reductive metric $g$ of the Berger 
sphere $\mathscr{M}^{3}(c)$ studied 
in this article satisfies the condition 
$\pounds_{\xi}g=0$. Here $\pounds_{\xi}$ is the 
Lie differentiation by $\xi$. Namely the Reeb vector field 
of the Berger sphere $\mathscr{M}^{3}(c)$ is a unit Killing 
vector field. A compatible metric $g$ of a contact $3$-manifold $(M,\eta)$ is said to be $K$-\emph{contact metric} if 
it satisfies $\pounds_{\xi}g=0$. To look for nice compatible metric for prescribed 
contact form $\eta$ on (compact) contact $3$-manifolds, Chern and Hamilton \cite{Ch-Ha} suggested 
to consider the variational problem
\[
\mathrm{E}_{\mathrm{CH}}(g)=\int_{M}|\pounds_{\xi}g|^{2}
\,\mathrm{d}v_{\eta}
\]
on the space of all compatible metrics. For recent progress on 
this variational problem, see \cite{MPS}.
}
\end{remark}

\subsection{Magnetic fields on homogeneous Riemannian spaces}\label{sec:1.5}
Let $M=L/H$ be a homogeneous Riemannian space with $L$ compact and semi-simple. Denote by $\mathsf{B}=\mathsf{B}_{\mathfrak l}$ the Killing form of the Lie algebra $\mathfrak{l}$. 
Then, as we saw in section \ref{sec:1.1}, for any 
positive constant $\alpha$, 
$\langle\cdot,\cdot\rangle_{\mathsf B}=(-\alpha)\mathsf{B}$ is 
an $\mathrm{Ad}(L)$-invariant inner product 
on $\mathfrak{l}$, and hence, induces a bi-invariant 
Riemannian metric on $L$. Moreover, 
the tangent space $T_{o}M$ at the origin is identified 
with the orthogonal complement $\mathfrak{p}=\mathfrak{h}^{\perp}$ of 
$\mathfrak{h}$ with respect to $\langle\cdot,\cdot\rangle_{\mathsf B}$. 
The orthogonal decomposition $\mathfrak{l}=\mathfrak{h}\oplus
\mathfrak{p}$ is reductive. The resulting 
homogeneous Riemannian space $M=L/H$ is normal. 
We denote the inner product 
on $\mathfrak{p}$ induced from $\langle\cdot,\cdot\rangle_{\mathsf B}$ by the 
same letter. 
Let $\mathcal{I}$ be the moment of inertia tensor field of $L$. 
Then we obtain an endomorphism $\mathcal{I}_{\mathfrak p}$ of 
$\mathfrak{p}$ defined by
\[
\langle X,Y\rangle_{\mathsf B}=\langle \mathcal{I}_{\mathfrak p}X,Y\rangle_{\mathsf B},
\quad X,Y\in\mathfrak{p}.
\]
Since $\mathcal{I}_\mathfrak{p}$ is self-adjoint with respect to 
$\langle\cdot,\cdot\rangle_{\mathsf B}$ and $\mathrm{Ad}(H)$-invariant, 
$\mathcal{I}_{\mathfrak p}$ is diagonalizable and its 
eigenspaces are $\mathrm{Ad}(H)$-invariant and orthogonal with 
respect to $\langle\cdot,\cdot\rangle_{\mathsf B}$. 

Take an element $\zeta\in \mathfrak{z}(\mathfrak{h})$ of the center 
$\mathfrak{z}(\mathfrak{h})$ of the isotropy algebra $\mathfrak{h}$
and introduce an $\mathrm{Ad}(H)$-invariant $2$-form $\omega^{\zeta}$ on $M$ 
by \cite[(5)]{BJ} (see also \cite[(3)]{AvS}):
\begin{equation}\label{eq:homF}
F^{\zeta}_{o}(X,Y)=-\langle \zeta,[X,Y]\rangle_{\mathsf B},\quad 
X,Y\in\mathfrak{p}.
\end{equation}
One can see that $F^{\zeta}$ is a closed $2$-form on $M$. 
We call $F^{\zeta}$ the \emph{standard invariant magnetic field} of $L/H$. 
We denote by $\varphi^{\zeta}$ the Lorentz force 
associated with the magnetic field $F^{\zeta}$.
Bolsinov and Jovanovi{\'c} \cite{BJ} studied 
magnetic trajectories with respect to the standard invariant magnetic fields 
on normal homogeneous spaces.

Note that when $M$ is an $\mathrm{Ad}(L)$-orbit, then 
$F^{\zeta}$ is nothing but the 
\emph{Kirillov-Kostant-Souriau symplectic form} (see \cite{AM,AvBook,Besse}).

\section{The Berger 3-sphere}\label{section1}

In this section, we give an explicit matrix group model of
the Berger $3$-sphere $\mathscr{M}^3(c)$ equipped with a canonical left invariant 
contact structure.

\subsection{Contact magnetic field of $\mathbb{S}^3$}\label{sec:2.1}
As is well known, the unit $3$-sphere 
$\mathbb{S}^3$ is identified with the special
unitary group 
\[
G=\mathrm{SU}(2)=
\{P\in\mathrm{SL}_{2}\mathbb{C}~|~ \overline{{}^t\!P} P=\Vec{1}\}
\]
with bi-invariant Riemannian metric of constant curvature $1$. Here $\Vec{1}$ denotes 
the identity matrix.

Let us denote the Lie algebra $T_{\Vec{1}}G$ of $G$ by $\mathfrak{g}$ (or $\mathfrak{su}(2)$).
The bi-invariant metric $g_1$ of constant curvature $1$
on $G$ is induced by the following inner
product $\langle \cdot , \cdot \rangle_1$
on $\mathfrak{g}$:
\[
\langle X,Y \rangle_1=-\frac{1}{8}\mathsf{B}(X,Y)=-\frac{1}{2}\ \mathrm{tr} (XY),\quad 
X,Y \in \mathfrak{g}.
\]
Here $\mathsf{B}$ denotes the Killing form of $\mathfrak{g}$. 
We call $g_1$ the \emph{normalized Killing metric}. 
Take a \emph{quaternionic basis} of $\mathfrak{g}$:
\[
\Vec{i}=
\left (
\begin{array}{cc}
0 & \sqrt{-1} \\
\sqrt{-1} & 0
\end{array}
\right ),
\quad 
\Vec{j}=
\left (
\begin{array}{cc}
0 & -1 \\
1 & 0
\end{array}
\right ),
\quad 
\Vec{k}=
\left (
\begin{array}{cc}
\sqrt{-1} & 0 \\
0 & -\sqrt{-1}
\end{array}
\right ).
\]
The Lie group $\mathrm{SU}(2)$ is described as
\[
\mathrm{SU}(2)=
\left \{
\
\left (
\begin{array}{cc}
x_0+\sqrt{-1}~x_3 &
-x_2+\sqrt{-1}~x_1
\\
x_2+\sqrt{-1}~x_1
&
x_0-\sqrt{-1}~x_3
\end{array}
\right )
\quad\biggr \vert
\quad
x_0^2+x_1^2
+x_2^2+x_3^2=1
\right \}.
\]
In the spinor representation of the Euclidean $3$-space $\mathbb{E}^3$,
we identify $\mathbb{E}^3$ with $\mathfrak{g}=\mathfrak{su}(2)$
via the correspondence
\[
(x_1,x_2,x_3)\longleftrightarrow
x_1\Vec{i}+x_2\Vec{j}+x_3\Vec{k}=
\left(
\begin{array}{rr}
\sqrt{-1} ~x_3&  -x_2+\sqrt{-1}~x_1\\[2mm]
x_2+\sqrt{-1}~ x_1& -\sqrt{-1}~x_3
\end{array}
\right).
\]
The cross product $\times$ of $\mathbb{E}^3$ is related to the Lie bracket $[\cdot,\cdot]$ of $\mathfrak{g}$ are related by
\[
\Vec{x}\times\Vec{y}=\frac{1}{2}[\Vec{x},\Vec{y}]
\]
for $\Vec{x}=x_1\Vec{i}+x_2\Vec{j}+x_3\Vec{k}$ and 
$\Vec{y}=y_1\Vec{i}+y_2\Vec{j}+y_3\Vec{k}$.

Denote the left translated vector fields of $\{\Vec{i},
\Vec{j},\Vec{k}\}$ by $\{E_1,E_2,E_3\}$. The commutation relations of
$\{E_1,E_2,E_3\}$ are
\[
[E_1,E_2]=2E_3,\quad
[E_2,E_3]=2E_1,\quad
[E_3,E_1]=2E_2.
\]
The left invariant $1$-form 
\[
\eta_1=g_{1}(E_3,\cdot )
\]
is a \emph{contact form} with 
Reeb vector field 
\[
\xi_1:=E_3.
\]
The metric $g_1$ is compatible to $\eta_1$ and 
the Lorentz force $\varphi_1$ is given by
\[
\varphi_1(E_1)=-E_2,\quad 
\varphi_1(E_2)=E_1,\quad 
\varphi_1(E_3)=0.
\]
It should be remarked that $\xi_1$ is a unit Killing vector field. 
The magnetic field $F_1=\mathrm{d}\eta_1$ will be 
referred to as the \emph{standard contact magnetic field} of 
$\mathbb{S}^3$.
The Levi-Civita connection $\nabla^1$ of the metric $g_1$ satisfies 
\[
(\nabla^{1}_{X}\varphi)Y
=-g_{1}(X,Y)\xi_{1}+\eta_{1}(Y)\xi_{1},
\quad 
\nabla^{1}_{X}\xi_{1}=\varphi X.
\]

The Lie group $G$ acts isometrically on the Lie algebra
$\mathfrak{g}$ by the $\mathrm{Ad}$-action.
\[
\mathrm{Ad}:G \times \mathfrak{g}
\rightarrow \mathfrak{g};
\>\> \mathrm{Ad}(a)X=aXa^{-1},
\quad a \in G,\ X \in \mathfrak{g}.
\]
The $\mathrm{Ad}$-orbit of $\Vec{k}/2$ is a $2$-sphere 
$\mathbb{S}^2(1/2)$ of radius
$1/2$ in the Euclidean $3$-space $\mathbb{E}^3=\mathfrak{g}$. The
$\mathrm{Ad}$-action of $G$ on $\mathbb{S}^2(1/2)$ is isometric and
transitive. The isotropy subgroup of $G$ at $\Vec{k}/2$ is
\[
K_1=\left \{
\left.
\exp(t\Vec{k})=
\left (
\begin{array}{cc}
e^{\sqrt{-1}t}
& 0
\\ 0 &
e^{-\sqrt{-1}t}
\end{array}
\right )
\ \right \vert
\ t \in {\mathbb R}
\right \}\cong \mathrm{U}(1) 
=\{e^{\sqrt{-1}t}\>|\>t\in\mathbb{R}\}.
\]
Hence $\mathbb{S}^2(1/2)$ is represented by 
$G/K_1=\mathrm{SU}(2)/\mathrm{U}(1)$
as a Riemannian symmetric space. The natural projection
\[
\pi_1:{\mathbb{S}}^3 \rightarrow \mathbb{S}^2(1/2),
\quad\pi_1(a)=\mathrm{Ad}(a)(\Vec{k}/2)
\]
is a Riemannian submersion and
defines a principal $\mathrm{U}(1)$-bundle
over $\mathbb{S}^2(1/2)$. This fibering is nothing 
but the well known \emph{Hopf fibering}.
Moreover it is the \emph{Boothby-Wang fibering} of 
$\mathbb{S}^3$ as a regular contact $3$-manifold \cite{BW}.

\subsection{The unit tangent sphere bundle $\mathrm{U}\mathbb{S}^3$}
\label{sec:2.2new}
For each $a\in G$, we may identify the 
representation matrix 
of the orthogonal transformation $\mathrm{Ad}(a):
\mathfrak{g}\to\mathfrak{g}$
with respect to the basis $\{\Vec{i},\Vec{j},\Vec{k}\}$
and the orthogonal transformation $\mathrm{Ad}(a)$ itself. 
Then we obtain a double covering 
$\mathrm{Ad}:\mathrm{SU}(2)\to\mathrm{SO}(3)$. 
The differential map $\mathrm{Ad}_{*\Vec{1}}:
\mathfrak{g}\to\mathfrak{so}(3)$ of $\mathrm{Ad}$ at $\Vec{1}$ is 
a Lie algebra isomorphism from $\mathfrak{su}(2)$ to 
$\mathfrak{so}(3)$ and 
given explicitly by
\begin{equation}\label{eq:su(2)so(3)}
\mathrm{Ad}_{*\Vec{1}}\Vec{i}=2\Vec{e}_{1},
\quad 
\mathrm{Ad}_{*\Vec{1}}\Vec{j}=2\Vec{e}_{2},
\quad 
\mathrm{Ad}_{*\Vec{1}}\Vec{k}=2\Vec{e}_{3}
\in\mathfrak{so}(3).
    \end{equation}
Obviously, $\mathrm{Ad}:\mathrm{SU}(2)\to\mathrm{SO}(3)$ is 
\emph{not} a Riemannian submersion.

On the other hand, the unit tangent sphere bundle 
\[
\mathrm{U}\mathbb{S}^2(1/2)=\{(X,V)\in\mathfrak{g}\times\mathfrak{g}
\>|\>
\langle X,X\rangle_1=\tfrac{1}{4},\>\>\langle V,V\rangle_1=1,
\>\>\langle X,V\rangle_1=0\}
\]
of $\mathbb{S}^2(1/2)$ is a homogeneous space of $G=\mathrm{SU}(2)$. The 
natural projection $\pi_{2}:G\to\mathrm{U}\mathbb{S}^2(4)$ is given 
explicitly by (\textit{cf}. \cite{IM19}):
\[
\pi_{2}(a)=(\mathrm{Ad}(a)\Vec{k}/2,\mathrm{Ad}(a)\Vec{i}).
\]
The isotropy subgroup of $G$ at $(\Vec{k}/2,\Vec{i})$ is 
$\{\pm\Vec{1}\}\cong\mathbb{Z}_2$. Hence 
we obtain the identification $\mathrm{U}\mathbb{S}^2(1/2)=\mathrm{SO}(3)$. 

The projection $\pi_{2}$ is regarded as a 
double covering $\mathrm{SU}(2)\to\mathrm{SO}(3)$ and 
satisfies 
\[
\pi_{2*\Vec{1}}\Vec{i}=\Vec{e}_{1},
\quad 
\pi_{2*\Vec{1}}\Vec{j}=\Vec{e}_{2},
\quad 
\pi_{2*\Vec{1}}\Vec{k}=\Vec{e}_{3}
\in\mathfrak{so}(3).
\]
Thus $\pi_{2}$ is a Riemannian submersion, but 
its differential map $\pi_{2*\Vec{1}}$ is 
\emph{not} a Lie algebra isomorphism.

\begin{remark}{\rm 
In \cite[III.1.4]{CB}, 
the following identification is used 
\[
\mathfrak{su}(2)\ni
x_{1}\Vec{i}
+x_{2}\Vec{j}
+x_{3}\Vec{k}
\longleftrightarrow
-2\left(
\begin{array}{ccc}
0 & -x_3 & x_2\\
x_3 & 0 & -x_1\\
-x_2 & x_1 & 0
\end{array}
\right)
\in\mathfrak{so}(3).
\]}
\end{remark}

\subsection{The Berger $3$-sphere}
For any real number $c>-3$, we deform the 
Riemannian metric $g_1$ of the unit $3$-sphere $\mathbb{S}^3$ as
\[
g(X,Y)
=\frac{4}{c+3}
\left(
g_{1}(X,Y)-
\frac{c-1}{c+3}\,
\eta_1(X)\eta_1(Y)
\right).
\]
The resulting Riemannian $3$-manifold 
$\mathscr{M}^3(c)=(\mathbb{S}^3,g)$ is called 
the \emph{Berger sphere}.

Precisely speaking, the original one due to Berger 
\cite{Berger} is $(\mathbb{S}^3,\frac{c+3}{4}\,g)$ and 
$c\not=1$. Note that under the limit 
$c\to -3$ in Gromov-Hausdorff sense, 
$(\mathbb{S}^3,\frac{c+3}{4}\,g)$ converges to
$\mathbb{S}^3$ equipped with the Carnot-Carath{\'e}odory metric. 
On the other hand, under the limit  $c\to 1$, 
$(\mathbb{S}^3,\frac{c+3}{4}\,g)$ collapses to $\mathbb{S}^2$.
For more Riemannian geometric studies on Berger spheres, we refer to \cite{Lau,Pod,Ra,Sakai81}.

Another geometric property of $\mathscr{M}^3(c)$ is a relation to 
the geometry of \emph{isoparametric hypersurfaces}. 
One can see that $\mathscr{M}^{3}(-2)$ is isometric to the 
universal covering of 
the minimal Cartan hypersurface of the unit $4$-sphere (see \cite{FP}). 
The minimal Cartan hypersurface is realized as 
$\mathrm{SO}(3)/(\mathbb{Z}_2\times\mathbb{Z}_2).$
On the other hand, $\mathscr{M}^{3}(c)$ is realized as a 
geodesic sphere of radius $\tan^{-1}(\sqrt{c-1}/2)$ 
[resp. $\tan^{-1}(\sqrt{1-c}/2)$] in 
the complex projective plane 
$\mathbb{C}P^2(c-1)$ of constant 
holomorphic sectional 
curvature $c-1$ 
[resp. complex hyperbolic plane 
$\mathbb{C}H^2(c-1)$ of constant 
holomorphic sectional 
curvature $c-1$ ] when $c>1$ 
[resp. $-3<c<1$] (\textit{cf.} \cite{W}). 

Let us deform the contact form, 
Reeb vector field and Lorentz force of $\mathbb{S}^3$ as 
\[
\eta:=\frac{4}{c+3} \eta_1,\quad
\xi:=\frac{c+3}{4}\xi_1,\quad
\varphi:=\varphi_1.
\]
Then $g$ is compatible to $\eta$. 
The Berger sphere $\mathscr{M}^3(c)$ with $c\not=1$ is no longer a 
space form, but the sectional curvatures 
$K(X\wedge \varphi X)$ where $X\perp \xi$ is constant $c$.
The sectional curvature $K(X\wedge \varphi X)$ for $X\perp \xi$ is called 
the \emph{holomorphic sectional curvature}. Thus 
$\mathscr{M}^3(c)$ is of constant holomorphic sectional 
curvature $c$.
Note that the contact structure $\mathcal{D}$ of $\mathbb{S}^3$ is invariant under this deformation.

The Reeb vector field $\xi$ generates a one parameter group of
transformations on $\mathscr{M}^3(c)$. Since $\xi$ is a Killing
vector field with respect to the Berger metric, 
this transformation group acts isometrically on $G=\mathrm{SU}(2)$.
The transformation group generated by $\xi$ is identified with the
following Lie subgroup $K=K_c$ of $G$:
\[
K=\left\{\left.\exp\left(\frac{(c+3)t}{4}\,\Vec{k}\right)\>
\right|\>t\in\mathbb{R}
\right\}
\cong
\mathrm{U}(1).
\]
Furthermore, the action of the transformation group generated by
$\xi$ corresponds to the natural right action of $K$ on $G$:
\[
G \times K \rightarrow G;\
(a,k) \mapsto ak.
\]
By using the well-known curvature formula for Riemannian submersion 
due to O'Neill, one can
see that the orbit space $G/K$ is a $2$-sphere 
$\mathbb{S}^2(1/\sqrt{c+3})$ of radius
$1/\sqrt{c+3}$, namely the  constant curvature $(c+3)$-sphere. 
The projection $\pi:G\to\mathbb{S}^2(1/\sqrt{c+3})$ is 
explicitly given by
\[
\pi(a)=\mathrm{Ad}(a)(\Vec{k}/\sqrt{c+3}).
\]
The Riemannian metric $g$ is not
only $G$-left invariant but also $K$-right invariant. Hence
$G\times K$ acts isometrically on $G$. The Berger sphere 
$\mathscr{M}^3(c)$ is represented by $(G\times K)/\Delta K=G$ as a
naturally reductive homogeneous space. For $c\not=1$,
$\mathscr{M}^3(c)$ has $4$-dimensional isometry group. 
For more details, see Section \ref{sec:2.4}.

In particular, $g$ is $G$-bi-invariant if and only if $c=1$. In this
case $\mathscr{M}^3(1)$ is represented by $(G\times G)/\Delta G$ as a
Riemannian symmetric space. Note that $\mathscr{M}^3(1)$ has
$6$-dimensional isometry group.

The deformed structure $(\varphi,\xi,\eta,g)$ is still left invariant 
on $G=\mathrm{SU}(2)$. The inner product 
$\langle\cdot,\cdot\rangle$ on the Lie algebra 
$\mathfrak{g}=\mathfrak{su}(2)$ induced from the 
deformed metric $g$ is related to the standard 
inner product $\langle\cdot,\cdot\rangle_{1}=-\mathsf{B}/8$ via the 
moment of inertia tensor field $\mathcal{I}$ 
defined by
\[
\mathcal{I}\Vec{i}=I_{1}\Vec{i},
\quad 
\mathcal{I}\Vec{j}=I_{2}\Vec{j},
\quad 
\mathcal{I}\Vec{i}=I_{3}\Vec{k},
\quad 
I_{1}=I_{2}=\frac{4}{c+3},
\quad 
I_{3}=\frac{16}{(c+3)^2}.
\]

\subsection{The Levi-Civita connection}

Consider an orthonormal frame field $\{e_1,e_2,e_3\}$ of
$\mathscr{M}^3(c)$ by
\[
e_1:=\frac{\sqrt{c+3}}{2}E_1,\quad 
e_2:=\frac{\sqrt{c+3}}{2}E_2,\quad 
e_3:=\frac{c+3}{4}\xi_1.
\]
Then the commutation relations of this basis are
\[
[e_1,e_2]=2e_3,\quad 
[e_2,e_3]=\frac{c+3}{2}e_1,,\quad  
[e_3,e_1]=\frac{c+3}{2}e_2.
\]
The Levi-Civita connection $\nabla$ of $(\mathscr{M}^3(c),g)$ is
described by
\[
\nabla_{e_1}e_1=0,\quad 
\nabla_{e_1}e_2=e_3,\quad 
\nabla_{e_1}e_3=-e_2,
\]
\[
\nabla_{e_2}e_1=-e_3,\quad 
\nabla_{e_2}e_2=0,\quad 
\nabla_{e_2}e_3=e_1,
\]
\[
\nabla_{e_3}e_1=\frac{c+1}{2}e_2,\quad 
\nabla_{e_3}e_2=-\frac{c+1}{2}e_1,\quad 
\nabla_{e_3}e_3=0.
\]
The Riemannian curvature
tensor field $R$ of $(\mathscr{M}^{3}(c),g,\nabla)$
is described by
\begin{equation*}
R_{1212}=c,\quad  R_{1313}=R_{2323}=1,
\end{equation*}
and the sectional curvatures are:
\begin{equation*}
K_{12}=c,\quad 
K_{13}=K_{23}=1.
\end{equation*}
The Ricci tensor field $\mathrm{Ric}$ and the
scalar curvature $\mathrm{scal}$ are computed to be
\begin{equation*}
\mathrm{Ric}_{11}=\mathrm{Ric}_{22}
=c+1,\quad  \mathrm{Ric}_{33}=2,\quad \mathrm{scal}=2(c+2).
\end{equation*}
 
The symmetric tensor $\mathsf{U}$, defined by \eqref{eq:U}, has the following 
essential components
\[
\mathsf{U}(e_1,e_3)=\frac{c-1}{4}e_2,
\quad 
\mathsf{U}(e_2,e_3)=-\frac{c-1}{4}e_1,
\]
other components being zero.

The covariant derivatives of $\varphi$ and $\xi$ are given by
\[
(\nabla_{X}\varphi)Y=-g(X,Y)\xi+\eta(Y)X,
\quad 
\nabla_{X}\xi=\varphi X.
\]

\subsection{Homogeneous structure}
\label{sec:2.4}
As we saw before $G$ acts isometrically and transitively 
on $\mathbb{S}^2(1/\sqrt{c+3})$ via the Ad-action. The Lie algebra 
$\mathfrak{k}/\sqrt{c+3}$ of the isotropy subgroup $K$ is 
$\mathbb{R}\Vec{k}$. The tangent space 
$T_{\Vec{k}}\mathbb{S}^2(1/\sqrt{c+3})$ is identified with
$\mathfrak{m}=\{x_1\Vec{i}+x_2\Vec{j}\>|\>x_1,x_2\in\mathbb{R}\}$. We have the orthogonal 
direct sum decomposition $\mathfrak{g}=\mathfrak{k}\oplus\mathfrak{m}$. 
The linear subspace $\mathfrak{m}$ induces a 
left invariant contact structure 
on $G$.

The Riemannian metric $g$ of $\mathscr{M}^{3}(c)$ is invariant 
under the action of $G\times K$:
\[
(G\times K)\times G\to G;\>\>((a,k),x)\longmapsto axk^{-1}.
\]
The isotropy subgroup at $\Vec{1}$ is 
\[
\Delta K=\{(k,k)\>|\>k\in K\}\cong K.
\] 
Hence we obtain a homogeneous Riemannian space 
representation $\mathscr{M}^3(c)=(G\times K)/\Delta K$. 
Note that when $c=1$, the metric $g_1$ is 
$G\times G$-invariant and 
$\mathbb{S}^3=(G\times G)/\Delta G$ is a Riemannian symmetric space.

The tangent space $T_{\Vec{1}}\mathscr{M}^3(c)$ 
of $\mathscr{M}^3(c)$ at $\Vec{1}$ is 
identified with 
the linear subspace
\[
\mathfrak{p}(c)=\left\{\left.
\left(
V+W,\frac{1-c}{4}~W
\right)
\>\right|\>V\in\mathfrak{m},\>W\in\mathfrak{k}\right\}.
\]
The Lie algebra of $G\times K$ is $\mathfrak{g}\oplus\mathfrak{k}$ and
the Lie algebra of $\Delta K$ is given by
\[
\Delta \mathfrak{k}
=\{(W,W)\>|\>W\in\mathfrak{k}\}
\cong\mathfrak{k}.
\]
We have the decomposition
\[
\mathfrak{g}\oplus\mathfrak{k}=\triangle \mathfrak{k}
\oplus \mathfrak{p}(c).
\]
Every element $(X,Y)\in\mathfrak{g}\oplus\mathfrak{k}$ is 
decomposed as
\[
\begin{array}{rcl}
(X,Y) & = & \quad \displaystyle
\left(
\frac{c-1}{c+3}X_{\mathfrak k}+\frac{4}{c+3}Y, ~
\frac{c-1}{c+3}X_{\mathfrak k}+\frac{4}{c+3}Y
\right)\\[4mm]
& & \displaystyle +\left(
X_{\mathfrak{m}}+
\frac{4}{c+3}(X_{\mathfrak{k}}-Y),
-\frac{c-1}{c+3}(X_{\mathfrak{k}}-Y)
\right).
\end{array}
\] 
In particular we have
\[
(X,0)=
\left(
\frac{c-1}{c+3}X_{\mathfrak k}, ~
\frac{c-1}{c+3}X_{\mathfrak k}
\right)
+\left(
X_{\mathfrak{m}}+
\frac{4}{c+3}X_{\mathfrak{k}},
-\frac{c-1}{c+3}X_{\mathfrak{k}}
\right).
\] 
Thus for any tangent vector 
$X\in\mathfrak{g}=T_{\Vec{1}}\mathscr{M}^{3}(c)$, 
the corresponding vector 
$\widehat{X}\in\mathfrak{p}(c)$ is explicitly given by
\[
\widehat{X}=\left(
X_{\mathfrak{m}}+
\frac{4}{c+3}X_{\mathfrak{k}},
-\frac{c-1}{c+3}X_{\mathfrak{k}}
\right).
\]
The Lie algebra $\mathfrak{l}$ is spanned by 
the orthonormal basis
\begin{equation}\label{eq:l-basis}
\{
(e_1,0),(e_2,0),(e_3,0),(0,e_3)
\}.
\end{equation}

It is straightforward to prove that the tensor $\mathsf{U}_\mathfrak{p}$, associated to $\mathfrak{p}$ and
defined by \eqref{eq:NatRed}, vanishes. 
Namely, $\mathscr{M}^{3}(c)=(G
\times K)/K$ is naturally reductive.

\medskip

The exponential map 
$\exp_{G\times K}:\mathfrak{g}\oplus\mathfrak{k}
\to G\times K$ is given explicitly by
\begin{equation}\label{eq:expGK}
\exp_{G\times K}\widehat{X}=
\left(
\exp_{G}\left(
X_{\mathfrak m}+\frac{4}{c+3}X_{\mathfrak k}
\right),
\exp_{K}\left(
\frac{c-1}{c+3}X_{\mathfrak k}
\right)
\right)
\end{equation}
for any $\widehat{X}\in\mathfrak{p}$ corresponding 
to $X=X_{\mathfrak k}+X_{\mathfrak m}\in\mathfrak{g}$.

\medskip

\begin{remark}{\rm
In case $c=1$, the Riemannian symmetric space 
$G=(G\times G)/\Delta G$ admits a reductive decomposition
\[
\mathfrak{g}\oplus\mathfrak{g}
=\triangle\mathfrak{g}+\mathfrak{q},
\quad 
\mathfrak{q}=\{(X,-X)\>|\>X\in\mathfrak{g}\}\cong 
\mathfrak{g}.
\]
Every $(X,Y)\in\mathfrak{g}\oplus\mathfrak{g}$ is decomposed as
\[
(X,Y)=\left(\frac{X+Y}{2},\frac{X+Y}{2}
\right)+
\left(\frac{X-Y}{2},\frac{-X+Y}{2}
\right).
\]
}
\end{remark}

\begin{remark}{\rm
The original Berger sphere is represented as a homogeneous space 
of $\mathrm{SU}(2)\times\mathbb{R}$ (\cite{Chavel}, see also \cite{Sakai81}).
Moreover, the original Berger sphere can be represented as $\mathrm{U}(2)/\mathrm{U}(1)$. 
For these models, see Appendix \ref{sec:A} of the present article.}
\end{remark}
Let us pick up the normal homogeneous $3$-sphere
$\mathbb{S}^3=L/H=(\mathrm{SU}(2)\times\mathrm{U}(1))/\Delta \mathrm{U}(1)$. The inner product $\langle\cdot,\cdot\rangle_1$ of 
$\mathfrak{su}(2)$ is extended to $\mathfrak{l}$ by the rule 
so that \eqref{eq:l-basis} is orthonormal with respect to it. 
We denote the extended inner product by the same letter $\langle\cdot,\cdot\rangle_1$. Since $\mathfrak{z}(\mathfrak{h})=\Delta\mathfrak{u}(1)$, 
we may choose an element $\zeta\in \mathfrak{z}(\mathfrak{h})$ by
\begin{equation}\label{eq:homEta}
\zeta=\left(
\frac{1}{2}\xi_1,\frac{1}{2}\xi_1
\right).
\end{equation}
Then standard invariant magnetic field $F^{\zeta}$ defined by \eqref{eq:homF} 
coincides with the standard contact magnetic field $F=\mathrm{d}\eta$ of $\mathbb{S}^3$.


\section{The Euler-Arnold equation}
Since the Berger sphere $\mathscr{M}^{3}(c)
=(G\times K)/\Delta K$ is naturally reductive, 
every geodesic is an orbit of a one-parameter 
subgroup of $G\times K$. In this section we give an explicit 
representation for geodesics of the Berger sphere $\mathscr{M}^{3}(c)$.

\subsection{}
First we deduce the Euler-Arnold equation 
$\dot{\mu}-[\mu,\varOmega]=0$ for 
$\mathscr{M}^{3}(c)$. 
Let $\gamma(t)$ be an arc length parametrized curve 
in $\mathscr{M}^{3}(c)$. Represent the 
angular velocity as
\[
\varOmega(t)=A(t)e_1+B(t)e_2+C(t)e_3, \quad A(t)^2+B(t)^2+C(t)^2=1.
\]
It follows that the 
momentum $\mu=\mathcal{I}\varOmega$ is 
\[
\mu(t)=\frac{4}{c+3}
\left(
A(t)e_1+B(t)e_2+
\frac{4C(t)}{c+3}e_3
\right).
\]
Since 
\[
[\mu,\varOmega]
=
\frac{2(c-1)C(t)}{c+3}
\left(B(t)e_1-A(t)e_2\right),
\]
the Euler-Arnold equation is the ODE system
\begin{equation}
\label{eq:EA-eq-04}
\left\{
\begin{array}{l}
\dot{A}=~\frac{c-1}{2}~CB,
\\[2mm]
\dot{B}=-\frac{c-1}{2}~CA,
\\[2mm]
\dot{C}=0.
\end{array}
\right.
\end{equation}
From the third equation of the system, we get $C(t)=\cos\theta$ is a constant.
The coefficients $A(t)$ and $B(t)$ are obtained as
\[
\left(
\begin{array}{c}
A(t)
\\[2mm]
B(t)
\end{array}
\right)
=
\left(
\begin{array}{cc}
\cos\{t(\frac{c-1}{2}\cos\theta)\}
& 
\sin\{t(\frac{c-1}{2}\cos\theta)\}
\\[2mm]
-\sin\{t(\frac{c-1}{2}\cos\theta)\}
& 
\cos\{t(\frac{c-1}{2}\cos\theta)\}
\end{array}
\right)
\left(
\begin{array}{c}
A_0
\\[2mm]
B_0
\end{array}
\right)
\]
for some constants $A_0$ and $B_0$ satisfying 
$A_0^2+B_0^2=\sin^2\theta$.

When $c=1$, we obtain
\[
A(t)=A_0,\quad B(t)=B_0,\quad 
C(t)=\cos\theta.
\]
Hence $\varOmega(t)$ is a constant vector. 
Thus the solution of the ODE $\gamma(t)^{-1}\dot
{\gamma}(t)=\varOmega(t)$ under the 
initial condition 
$\gamma(0)=\Vec{1}$ is 
\[
\gamma(t)=\exp_{G}(tX),
\quad \textrm{where} \quad X=A_{0}e_1+B_{0}e_2+\cos\theta\,e_3\in
\mathfrak{g}.
\]
Thus we retrieve the well known fact 
that all geodesics of the Riemannian 
symmetric space $\mathbb{S}^3=(G\times G)/\Delta G$ are homogeneous. 

\begin{remark}{\rm Since $\mathbb{S}^3$ 
is identified with $G=\mathrm{SU}(2)$, the orbit 
of the identity $\Vec{1}$ under the $1$-parameter subgroup 
$\{\exp_{G}(tX)\}_{t\in \mathbb{R}}$ is 
\[
\exp_{G}(tX)\cdot \Vec{1}=\exp_{G}(tX).
\]
}
\end{remark}

\begin{proposition}
Let $\gamma(t)$ be a unit speed geodesic 
of $\mathscr{M}^{3}(c)$ with $c\not=1$ 
starting at $\Vec{1}$ with initial 
angular velocity 
$\varOmega(0)=A_{0}e_1+B_{0}e_2+\cos\theta e_3$. 
Then the angular velocity 
$\varOmega(t)$ of $\gamma(t)$ is given by
\[
\varOmega(t)=A(t)e_1+B(t)e_2+\cos\theta e_3,
\]
where $\theta$ is a constant and 
\begin{equation}
\left\{
\begin{array}{l}
A(t)=~A_{0}\cos\{t(\frac{c-1}{2}\cos\theta)\}
+B_{0}\sin\{t(\frac{c-1}{2}\cos\theta)\},\\[2mm]
B(t)=-A_{0}\sin\{t(\frac{c-1}{2}\cos\theta)\}
+B_{0}\cos\{t(\frac{c-1}{2}\cos\theta)\}.
\end{array}
\right.
\end{equation}
\end{proposition}
It should be remarked that 
every unit speed geodesic $\gamma(t)$ makes 
constant angle $\theta$ with the Reeb flow.
In other words, the angle $\theta$ (called the 
\emph{contact angle}) of a unit speed geodesic 
is a first integral.

\subsection{}
Hereafter we assume that $c\not=1$. 
We wish to solve 
the ODE $\dot{\gamma}(t)=\gamma(t)\varOmega(t)$ 
under the initial condition 
$\gamma(0)=\Vec{1}$. 
For this purpose here we prove the 
following lemma.
\begin{lemma}\label{lem:geodesiclemma}
A curve 
$\gamma_{W,V}(t)=\exp_{G}(tW)\exp_{K}(-tV)$, where 
$W=ue_1+ve_2+we_3\in\mathfrak{g}$ and $V=\sigma e_3\in\mathfrak{k}$ is 
a geodesic in $\mathscr{M}^{3}(c)$ with $c\not=1$ if and 
only if $\sigma=(1-c)w/4$.
\end{lemma}
\begin{proof}
The angular velocity 
$\varOmega(t)$ of 
$\gamma_{W,V}(t)$ is 
computed as
\[
\varOmega(t)=
\mathrm{Ad}(\exp_{K}(tV))W-V.
\]
One can deduce that
\begin{align*}
& \mathrm{Ad}(\exp_{K}(tV))e_1
=~\cos\frac{(c+3)\sigma\,t}{2}e_1 +\sin\frac{(c+3)\sigma\,t}{2}e_2,\\
& \mathrm{Ad}(\exp_{K}(tV))e_2
=-\sin\frac{(c+3)\sigma\,t}{2}e_1+\cos\frac{(c+3)\sigma\,t}{2}e_2,\\
& \mathrm{Ad}(\exp_{K}(tV))e_3
=e_3.
\end{align*}
Hence the angular velocity $\varOmega(t)=w_1(t)e_1
+w_2(t)e_2+w_3(t)e_3$ of $\gamma_{W,V}(t)$ is 
given by
\begin{align*}
w_1(t)=&u\cos\frac{(c+3)\sigma\,t}{2}-v\sin\frac{(c+3)\sigma\,t}{2},
\\
w_2(t)=&u\sin\frac{(c+3)\sigma\,t}{2}+
v\cos\frac{(c+3)\sigma\,t}{2},
\\
w_3(t)=&w-\sigma.
\end{align*}
From the Euler-Arnold equation, 
$\gamma_{W,V}(t)=\exp_{G}(tW)\exp_{K}(-tV)$ is a geodesic  
if and only if 
\[
\dot{w}_1=\frac{c-1}{2}w_3w_2,
\quad
\dot{w}_2=-\frac{c-1}{2}w_3w_1,
\quad
\dot{w}_3=0.
\]
This system is explicitly 
computed as
\[
\sigma=\frac{1-c}{4}w.
\]
Thus $\gamma_{V,W}(t)$ is a geodesic 
when and only when $\sigma=(1-c)w/4$.
\end{proof}
From this lemma, 
a geodesic $\gamma(t)$ of the form 
$\gamma(t)=\exp_{G}(tW)\exp_{K}(-tV)$ 
is rewritten as
\begin{align*}
\gamma(t)=&
\exp_{G}(tW)\exp_{K}(-tV)
\\
=&\exp_{G}\{t(ue_1+ve_2+we_3)\}
\exp_{K}\left\{
-t\left(
\frac{1-c}{4}w
\right)e_3
\right\}.
\end{align*}
Now let us set
\begin{equation}\label{eq:initialvector}
X_{\mathfrak m}:=ue_1+ve_2,
\quad 
X_{\mathfrak k}:=\frac{(c+3)w}{4}e_3.
\end{equation}
Then
\[
\exp_{G}\{t(ue_1+ve_2+we_3)\}
=\exp_{G}
\left\{
t
\left(
X_{\mathfrak{m}}+
\frac{4}{c+3}X_{\mathfrak{k}}
\right)
\right\}.
\]
and 
\[
\exp_{K}(-tV)
=\exp_{K}\left\{
-t\left(
\frac{1-c}{4}w
\right)e_3
\right\}
=
\exp_{K}\left\{
-t\left(
\frac{1-c}{c+3}w
\right)X_{\mathfrak{k}}
\right\}.
\]
Hence 
\[
\gamma(t)=\exp_{G}
\left\{t
\left(
X_{\mathfrak{m}}+
\frac{4}{c+3}X_{\mathfrak{k}}
\right)
\right\}
\,
\exp_{K}
\left\{
t\left(
\frac{c-1}{c+3}X_{\mathfrak k}
\right)
\right\}.
\]
Thus if we set 
\[
X=X_{\mathfrak k}+X_{\mathfrak m}\in
\mathfrak{g},
\]
then we get
\begin{equation}\label{eq:3.4new}
\gamma(t)=
\exp_{G}
\left\{t
\left(
X_{\mathfrak{m}}+
\frac{4}{c+3}X_{\mathfrak{k}}
\right)
\right\}
\,
\exp_{K}
\left\{
t\left(
\frac{c-1}{c+3}X_{\mathfrak k}
\right)
\right\}
=\exp_{G\times K}(t\widehat{X}),
\end{equation}
where $\widehat{X}$ is an element of $\mathfrak{p}$ corresponding 
to $X$. 
The formula \eqref{eq:3.4new} shows that $\gamma(t)$ is a homogeneous geodesic. 
Note that \eqref{eq:3.4new} is valid also for the case $c=1$.
Now we arrive at the stage to prove the 
following fundamental fact (see \eqref{eq:expGK} and \textit{cf.} \cite{DZ,Gordon}).

\begin{theorem}
\label{THM3.1}
The geodesic starting at the 
origin of $\mathscr{M}^3(c)$ with initial 
velocity 
$\widehat{X}\in\mathfrak{p}$ is given 
explicitly by
\begin{equation}\label{eq:homgeo}
\gamma(t)=
\exp_{G\times K}(t\widehat{X})=\exp_{G}
\left\{t
\left(
X_{\mathfrak{m}}+
\frac{4}{c+3}X_{\mathfrak{k}}
\right)
\right\}
\,
\exp_{K}
\left\{
t\left(
\frac{c-1}{c+3}X_{\mathfrak k}
\right)
\right\}.
\end{equation}
Here $X=X_{\mathfrak{k}}+X_{\mathfrak{m}}\in\mathfrak{g}$ 
corresponds to $\widehat{X}\in\mathfrak{p}$ under the 
identification $\mathfrak{p}=T_{\Vec{1}}\mathscr{M}^{3}(c)=\mathfrak{g}$.
\end{theorem}
\begin{proof}
Let $\gamma(t)$ be a curve of the form 
\eqref{eq:homgeo}. Then from 
Lemma \ref{lem:geodesiclemma}, 
$\gamma(t)$ is a geodesic 
starting at $\Vec{1}$ with initial velocity 
$\widehat{X}\in \mathfrak{p}$. 

Conversely let $\gamma(t)$ be a geodesic 
starting at $\Vec{1}$ with initial velocity 
$\widehat{X}\in \mathfrak{p}$.
Then 
$\gamma(t)$ is expressed as
$\gamma(t)=\exp_{G\times K}(t\hat{X})$. 
\end{proof}

\subsection{Some remarks on geodesics}
Let us make some remarks:
\begin{remark}[Legendre geodesics]{\rm
A curve $\gamma_X(t)=\exp_G(tX)$, where $X=ue_1+ve_2+we_3\in\mathfrak{g}$ is a geodesic in 
$\mathscr{M}^3(c)$ with $c\neq1$ if and only if it is a Legendre geodesic $\exp_G(t(ue_1+ve_2))$,
with $u^2+v^2=1$. This fact can be verified by the equation 
$\mathsf{U}(X,X)=0$. One can confirm that a Legendre geodesic $\exp_{G}(t(ue_1+ve_2))$ lies 
in the minimal $2$-sphere in $\mathscr{M}^{3}(c)$. 
}
\end{remark}

\begin{remark}[Legendre geodesics]{\rm
A curve 
$\gamma_{W,V}(t)=\exp_{G}(tW)\exp_{K}(-tV)$, where 
$W=ue_1+ve_2+we_3\in\mathfrak{g}$ and $V=\lambda e_3\in\mathfrak{k}$ is 
a Legendre geodesic in $\mathscr{M}^{3}(c)$ with $c\not=1$ if and only if
$W\in\mathfrak{m}$ and $V=0$.
}
\end{remark}

\begin{remark}[Reeb flows]
{\rm
The Reeb flow $\exp_G(\pm(te_3))$ in $\mathscr{M}^{3}(c)$ can be
also expressed as
\[
\exp_{G}\left\{\frac{4t}{c+3}e_3\right\}\,\exp_{K}\left\{
\frac{(c-1)t}{c+3}e_3\right\}.
\]
Note that the Reeb flows are periodic.}
\end{remark}

\begin{remark}[Periodicity]{\rm
A geodesic $\exp_{G}\{t(ue_1+ve_2+we_3)\}
\exp_{K}\left\{
-t\left(
\frac{1-c}{4}w
\right)e_3
\right\}$ in $\mathscr{M}^3(c)$ with $c\neq1$ is periodic if and only if
\[
\frac{\sqrt{(c+3)-(c-1)\cos^2\theta}}{(c+3)(1-c)}\in\mathbb{Q}.
\]
}
\end{remark}

\begin{remark}{\rm
The function $(c+3)-(c-1)\cos^{2}\theta$ plays an important role in the study of geometry of
geodesics and magnetic trajectories. In our recent paper \cite{IM23a}, we introduce
the function $\ell=\frac{(c+3)-(c-1)\cos^{2}\theta}{4}+q\cos\theta$ to study magnetic Jacobi fields
in Sasakian space forms.

On the other hand, in \cite{Engel},  Engel introduced the function $\lambda$ by
\[
\lambda(t):=(c+3)\sin^{2}\theta(t)+4\cos^{2}\theta(t)=(c+3)-(c-1)\cos^{2}\theta(t)
\]
to measure how the geodesics of certain homogeneous Riemannian $3$-manifolds intersect with Reeb flows. 
More precisely, in \cite[Corollary 3.8]{Engel} it is proved the following result:
Let $\gamma(t)$ be a geodesic 
with constant angle $\theta$. 
If $\lambda>0$ along $\gamma$, then 
$\gamma$ intersects with the Reeb flow  
through $\gamma(0)$ periodically. 
If $\lambda\leq 0$ along $\gamma$, then 
$\gamma$ intersects with the Reeb flow through 
$\gamma(0)$ only in $\gamma(0)$.
A periodicity criterion is given in \cite[Corollary 3.12]{Engel}.
}
\end{remark}

\begin{remark}{\rm The function $\lambda$ has an application to 
the study of shortest non-trivial closed geodesics.
In \cite[Corollary 3.13]{Engel} it is proved that: If $c>1$, the shortest 
non-trivial closed geodesics are given by Reeb flows. 
For the case $0<c<1$, the shortest non-trivial closed geodesics 
are given by Legendre geodesics. 
This fact is deduced from the inequality
\[
L(\gamma)\geq \frac{8\pi}
{2\sqrt{\lambda(t)}+(c-1)\cos\theta(t)}.
\]
Here $L(\gamma)$ is the length of a simply closed geodesic.

As a matter of fact, Engel pointed out that the extremal values of the sectional curvature are
$1$ and $c$. The Jacobi operator of a normal geodesic $\gamma(t)$ has eigenvalues
$1$ and $1+(c-1)\sin^{2}\theta$.
For a study on Jacobi fields of geodesics in Berger spheres, we refer to \cite{JW,Z}.
}
\end{remark}

\begin{remark}[Jacobi osculating rank]{\rm
 In \cite[Theorem 4.3]{G-D}, Gonz{\'a}lez-D{\'a}vila states that 
the Jacobi osculating rank of every geodesic 
is $2$ except Reeb flows. The Reeb flows have 
Jacobi osculating rank $0$.
  The conjugate points of a geodesic 
$\gamma(t)$ with constant angle $\theta$ to the origin 
are all of the form $\gamma(t/\sqrt{\lambda(t)})$, 
where $s\in\pi\mathbb{N}$ or 
$t$ is a solution to 
\[
\tan\frac{t}{2}=\frac{1-c}{8}\sin^{2}\theta\,t.
\] 
See \cite[Theorem 5.1]{G-D}. Finally, it is proved that
(in \cite[Proposition 5.3]{G-D}) any geodesic starting at the origin $\Vec{1}$ 
intersects the Reeb flow through $\Vec{1}$ exactly 
at its isotropic conjugate points.
}
\end{remark}

\subsection{Rigid bodies}\label{sec:RB-new}
Here we discuss relations 
between homogeneous geodesics in the 
Berger sphere $\mathscr{M}^3(c)$ and symmetric Euler-Arnold equation 
of torque free motions of rigid bodies.

Let us identify the Lie algebra 
$\mathfrak{g}=\mathfrak{su}(2)$ with 
$\mathfrak{so}(3)$ via the Lie algebra 
isomorphism \eqref{eq:su(2)so(3)}:
\[
x_{1}\frac{\Vec{i}}{2}+
x_{2}\frac{\Vec{j}}{2}+
x_{3}\frac{\Vec{k}}{2}
\longleftrightarrow 
\left(
\begin{array}{ccc}
0 & -x_{3} & x_{2}\\
x_{3} & 0 & -x_{1}\\
-x_{2} & x_{1} &0
 \end{array}
\right).
\]
As we saw in 
Example \ref{eg:RigidBody}, 
we identify $\mathfrak{so}(3)$ with 
$(\mathbb{R}^3,\times)$ through 
the Lie algebra isomorphism 
\eqref{eq:so(3)R3}. 
Under the identification 
$\mathfrak{su}(2)=(\mathbb{R}^3,\times)$, 
the angular velocity 
$\varOmega(t)$ is identified with a vector valued 
function $\Vec{\omega}(t)$. The momentum 
$\mu(t)$ is identified with a 
vector valued function 
$\Vec{\mu}(t)=\mathcal{I}\Vec{\omega}(t)$. 
The Euler-Arnold equation is 
rewritten as \eqref{eq:EA-SO3}. 
Then we can apply the formula \eqref{eq:symmetricEA} with 
$I_1=I_2=4/(c+3)$ and $I_3=16/(c+3)^2$ for 
$\Vec{\mu}(t)$. Obviously, the formula \eqref{eq:symmetricEA} 
has the same form of \eqref{eq:homgeo} under the 
replacements $G\mapsto \mathrm{SO}(3)$ and 
$K\mapsto\mathrm{SO}(2)$.

Thus we confirmed the following fact.

\begin{proposition}\label{prop:4.2}
Let $\gamma(t)$ be a homogeneous geodesics in the Berger sphere 
$\mathscr{M}^3(c)$
starting at $\Vec{1}$. Then its projection image 
in the rotation group $\mathrm{SO}(3)$ describes a motion of rigid body 
in Euclidean $3$-space with principal moments of inertia
\[
I_1=I_2=\frac{4}{c+3},
\quad 
I_3=\frac{16}{(c+3)^2}.
\]
\end{proposition}


\section{The magnetized Euler-Arnold equation}
\label{sec4}
Let us magnetize the Euler-Arnold equation on 
the Berger sphere $\mathscr{M}^3(c)$
by the contact magnetic field $F=\mathrm{d}\eta$. 

\subsection{}
Consider an arc length parametrized curve $\gamma(t)$ 
in $\mathscr{M}^3(c)$ starting at $\Vec{1}$. 
Since the Lorentz force $\varphi$ is left invariant, 
we have 
\[
\varphi\dot{\gamma}=\gamma \varphi\varOmega.
\]
Hence we get
\[
\varphi\varOmega=\mathcal{I}^{-1}
\mathcal{I}\varphi\varOmega
=\mathcal{I}^{-1}
(\mathcal{I}\varphi\varOmega).
\]
Thus we obtain
\[
\langle
\nabla_{\dot{\gamma}}\dot{\gamma}-q\varphi\dot{\gamma},
Z\rangle
=\langle
\mathcal{I}^{-1}(\dot{\mu}-[\mu,\varOmega]
-q\mathcal{I}\varphi\varOmega)
,Z\rangle
\]
for all $Z$. Thus the Lorentz equation is rewritten as
the \emph{magnetized Euler-Arnold equation}:
\begin{equation}\label{eq:Mag-EA}
\dot{\mu}-[\mu,\varOmega]
-q\mathcal{I}\varphi\varOmega=0.
\end{equation}
Decomposing the angular velocity $\varOmega$ as 
$\varOmega(t)=A(t)e_1+B(t)e_2+C(t)e_3$ we have
\[
\varphi\varOmega=B(t)e_1-A(t)e_2
\]
and hence
\[
\mathcal{I}\varphi\varOmega=
\frac{4}{c+3}(B(t)e_1-A(t)e_2).
\]
It follows that the magnetized Euler-Arnold equation is 
equivalent to the following ODE system:
\begin{equation}\label{eq:5.2}
\left\{
\begin{array}{l}
\dot{A}=~\big(q+\frac{c-1}{2}C\big)B,
\\[2mm]
\dot{B}=-\big(q+\frac{c-1}{2}C\big)A,
\\[2mm]
\dot{C}=0.
\end{array}
\right.
\end{equation}
Hence $C$ is a constant, say $C=\cos\theta$. 
Put $\tilde{q}=q+\frac{c-1}{2}\cos\theta$. 
Then the coefficients $A(t)$ and $B(t)$ are solutions of the ODE system
\[
\dot{A}(t)=\tilde{q}B(t),\quad 
\dot{B}(t)=-\tilde{q}A(t).
\]
Hence we obtain 
\[
A(t)=A_{0}\cos(\tilde{q}t)+B_{0}\sin(\tilde{q}t),
\quad 
B(t)=-A_{0}\sin(\tilde{q}t)+B_{0}\cos(\tilde{q}t),
\]
for some constants $A_0$ and $B_0$. Note that 
$A_0^2+B_0^2=\sin^2\theta$.

We want now to determine contact magnetic trajectories of the form
\[
\gamma(t)=\exp_{G}(tW)\exp_{K}(-tV),\quad \mbox{where}\>\> 
W=ue_1+ve_2+we_3\in\mathfrak{g} \>\>\mbox{and } V=\sigma e_3\in\mathfrak{k}.
\]

\medskip

We make the following notation
\[
\varOmega(t)=\gamma^{-1}(t)\dot{\gamma}(t):=w_1(t)e_1+w_2(t)e_2+w_3(t)e_3.
\]
Since $\varOmega(t)=\mathrm{Ad}(\exp_{K}(tV))W-V$ we obtain
\[
\left\{\begin{array}{l}
w_1(t)=u\cos\frac{(c+3)\sigma t}{2}-v\sin\frac{(c+3)\sigma t}{2},\\[2mm]
w_2(t)=u\sin\frac{(c+3)\sigma t}{2}+v\cos\frac{(c+3)\sigma t}{2},\\[2mm]
w_3(t)=w-\sigma.
\end{array}\right.
\]
The curve $\gamma(t)$ is a contact magnetic trajectory if and only if the magnetized 
Euler-Arnold equations are satisfied, that is
\[
\left\{
\begin{array}{l}
\dot{w}_1=~~\big(q+\frac{c-1}{2}w_3\big)w_2,\\[2mm]
\dot{w}_2=-\big(q+\frac{c-1}{2}w_3\big)w_1,\\[2mm]
\dot{w_3}=~~0.
\end{array}
\right.
\]
The third equation is fulfilled. The first two equations yield
\[
\left\{
\begin{array}{l}
~~\big(q+\frac{c-1}{2}w_3+\frac{(c+3)\sigma}{2}\big)
	\left(u\sin\frac{(c+3)\sigma t}{2}+v\cos\frac{(c+3)\sigma t}{2}\right)=0,\\[2mm]
-\big(q+\frac{c-1}{2}w_3+\frac{(c+3)\sigma}{2}\big)
	\left(u\cos\frac{(c+3)\sigma t}{2}-v\sin\frac{(c+3)\sigma t}{2}\right)=0.
\end{array}
\right.
\]
Therefore, $\gamma(t)$ is a contact magnetic trajectory if and only if
\[
q+\frac{c-1}{2}(w-\sigma)+\frac{(c+3)\sigma}{2}=0.
\]
This condition is equivalent to
\[
\sigma=\frac{1-c}{4}w-\frac{q}{2}.
\]
It follows that
\[
\gamma(t)=\exp_{G}\Big(t(ue_1+ve_2+we_3)\Big)
\exp_{K}\left(-t\left(\frac{1-c}{4}w-\frac{q}{2}\right)e_3\right).
\]
\begin{enumerate}
\item When $c=1$: 
we have
\[
\gamma(t)=\exp_{G}
\left(t(ue_1+ve_2+we_3)
\right)\,
\exp_{K}
\left\{
t\left(
\frac{q}{2}\xi
\right)
\right\}.
\]
Set $X=ue_1+ve_2+we_3\in\mathfrak{g}$, then 
$\gamma(t)$ is rewritten as
\[
\gamma(t)=\gamma_{X}(t)\,\exp_{K}
\left\{
t\left(
\frac{q}{2}\xi
\right)
\right\},
\]
where 
$\gamma_{X}(t)=\exp_{G}(tX)$ is a geodesic 
starting at the origin $\Vec{1}$ with 
initial velocity $X$. 
We call the curve $\exp_{K}
\left\{
t\left(
\frac{q}{2}\xi
\right)
\right\}$ a \emph{charged Reeb flow}. 
Thus we proved that the magnetic trajectory 
$\gamma(t)$ is a right translation of 
a homogeneous geodesic by the charged Reeb flow.
\item When $c\not=1$: 
In this case we define 
$X_{\mathfrak k}$ 
and $X_{\mathfrak m}$ 
by \eqref{eq:initialvector}. 
Then $\gamma(t)$ is rewritten as
\[
\gamma(t)=
\exp_{G\times K}(t\widehat{X})
\exp_{K}
\left\{
t\left(
\frac{q}{2}\xi
\right)
\right\}.
\]
\end{enumerate}
Here $\widehat{X}$ is the vector of $\mathfrak{p}$ corresponding to $X\in\mathfrak{g}$.

Henceforth we proved the main theorem of this article.
\begin{theorem}
The contact magnetic trajectory starting at the 
origin of the Berger $3$-sphere $\mathscr{M}^3(c)$ of holomorphic 
sectional curvature $c\not=1$ with initial 
velocity 
$\widehat{X}\in\mathfrak{p}$ 
is given 
explicitly by
\begin{equation}\label{eq:hom-cont-mag}
\gamma(t)=
\exp_{G\times K}(t\widehat{X})
\exp_{K}
\left\{
t\left(
\frac{q}{2}\xi
\right)
\right\}.
\end{equation}
\end{theorem}

We conclude this section with some observations.

\begin{remark}[Legendre magnetic trajectories]\rm
 A contact magnetic trajectory 
$\gamma_{W,V}(t)=\exp_{G}(tW)\exp_{K}(-tV)$, where 
$W=ue_1+ve_2+we_3\in\mathfrak{g}$ and $V=(\frac{1-c}{4}w-\frac{q}{2}) e_3\in\mathfrak{k}$ 
 in $\mathscr{M}^{3}(c)$ with $c\not=1$ is Legendre if and only if
 $w=-\frac{2q}{c+3}$.
\end{remark}
\begin{remark}[Periodicity]\rm 
The contact magnetic trajectory
$\gamma(t)=\gamma_{X}(t) \exp_{K}\{t(\frac{q}{2}\xi)\}$
in the unit sphere $\mathbb{S}^3$, where 
$\gamma_{X}(t)=\exp_{G}(tX)$ is a geodesic 
starting at the origin $\Vec{1}$ with 
initial velocity $X$, is periodic if and only if 
\[
\frac{q+\sqrt{q^2-4q\cos\theta+4}}{q-\sqrt{q^2-4q\cos\theta+4}}\in\mathbb{Q}.
\]
This is precisely the periodicity condition obtained in \cite[Theorem 6.1]{IM17}.
Here $\theta$ is the constant contact angle of the curve $\gamma$.
\end{remark}

\begin{remark}[Gyrostats]{\rm
In Proposition \ref{prop:4.2}, we observed that 
homogeneous geodesics in the Berger sphere 
$\mathscr{M}^3(c)$
starting at $\Vec{1}$ project down to 
curves in $\mathrm{SO}(3)$ which describe 
motions of rigid body 
in Euclidean $3$-space with principal moments of inertia
\begin{equation}\label{eq:InertiaGyro}
I_1=I_2=\frac{4}{c+3},
\quad 
I_3=\frac{16}{(c+3)^2}.
\end{equation}
Let us have a discussion about the projection image of the contact magnetic curves of $\mathscr{M}^{3}(c)$ in $\mathrm{SO}(3)$. 

We recall a generalization of 
Euler-Arnold equation on $\mathrm{SO}(3)$ by introducing a constant \emph{gyrostatic momentum}. 
As explained in \cite[\S 2.7]{BM}, this momentum 
(\emph{balanced gyrostat}) can be modeled, for example, 
by a balanced rotor rotating with constant angular velocity 
about an axis in the body. 
Let $\Vec{\kappa}\in\mathfrak{so}(3)\cong\mathbb{R}^3$ 
be the angular momentum of the rotor (called the 
\emph{constant gyrostat}). Then the equation of motion 
of a gyrostat \emph{without} (moment of) external forces 
is described as the following modified 
Euler-Arnold equation:
\begin{equation}\label{eq:GyroMag}
\dot{\Vec{\mu}}+\Vec{\omega}\times (\Vec{\mu}+\Vec{\kappa})=\Vec{0}.
\end{equation}
This equation is rewritten as 
\[
\mathcal{I}\dot{\Vec{\omega}}+\Vec{\omega}\times (
\mathcal{I}\Vec{\omega}+\Vec{\kappa})=\Vec{0}.
\]
Such a dynamical system is referred to as 
\emph{Zhukovskii-Volterra case} in \cite[\S 2.7]{BM}. 
Studies on Zhukovskii-Volterra case can be traced back to 
Volterra's work \cite{Volterra} and Zhukovskii's work 
\cite{Zhu}. 
Bolsinov and Borisov discussed 
Poisson geometric approach to Zhukovskii-Volterra case 
\cite[\S 4]{BB}.

Now let us reexamine the magnetized Euler-Arnold equation 
\eqref{eq:Mag-EA} under the identification $\mathfrak{su}(2)\cong 
\mathfrak{so}(3)\cong (\mathbb{R}^3,\times)$. 
The angular velocity 
$\varOmega(t)=A(t)e_1+B(t)e_2+C(t)e_3$ 
is identified with the vector valued 
function
\[
\Vec{\omega}(t)=
\sqrt{c+3}A(t)\Vec{e}_1
+\sqrt{c+3}B(t)\Vec{e}_2
+\frac{(c+3) C(t)}{2}\Vec{e}_3.
\]
as in Section \ref{sec:RB-new}. 
Hence the momentum $\mu(t)=\mathcal{I}\Vec{\omega}(t)$ is 
identified with the vector valued 
function 
\[
\Vec{\mu}(t)=(I_1\omega_1(t),I_1\omega_2(t),
I_3\omega_3(t))=
\frac{4}{\sqrt{c+3}}A(t)\Vec{e}_1
+\frac{4}{\sqrt{c+3}}B(t)\Vec{e}_2
+\frac{8C(t)}{c+3}\Vec{e}_3,
\]
where $I_1$, $I_2$ and 
$I_3$ are given by
\eqref{eq:InertiaGyro}. 
Hence we obtain
\[
\dot{\Vec{\mu}}(t)+\Vec{\omega}(t)\times\Vec{\mu}
(t)
=\frac{4}{\sqrt{c+3}}\left\{
\left(
\dot{A}-\frac{c-1}{2}CB
\right)\Vec{e}_1
+
\left(
\dot{B}+\frac{c-1}{2}CA
\right)\Vec{e}_2
+\dot{C}\Vec{e}_3
\right\}.
\]
Thus the Euler-Arnold equation \eqref{eq:EA-eq-04} 
is rewritten as 
$\dot{\Vec{\mu}}(t)+\Vec{\omega}(t)\times\Vec{\mu}
(t)=\Vec{0}$. 
One can check that  
the magnetized Euler-Arnold equation 
\eqref{eq:Mag-EA} is rewritten as the Euler-Arnold 
equation \eqref{eq:GyroMag} with 
constant gyrostat
\[
\Vec{\kappa}=\frac{4q}{c+3}\Vec{e}_3.
\]
Note that, under the identification \eqref{eq:su(2)so(3)}, 
the Reeb vector field $\xi_1$ of the standard contact form of the 
unit $3$-sphere $\mathbb{S}^3$ corresponds 
to the vector $\Vec{\xi}_1=2\Vec{e}_3$. Thus the gyrostat 
$\Vec{\kappa}$ is rewritten as $\Vec{\kappa}=\frac{2q}{c+3}\Vec{\xi}_1$.

Volterra \cite{Volterra} carried out the 
integration of the equation \eqref{eq:GyroMag} of motion 
in terms of elliptic functions (see also \cite{DGJ}). 
In case $I_1=I_2$, elliptic solutions reduce to 
trigonometric solutions. This fact is consistent with 
our main theorem. Indeed we gave a representation formula 
for magnetic trajectories in $\mathscr{M}^3(c)$ by using the 
exponential map.

Take two positive constants $c_1$ and $c_2$ and set 
\[
\check{e}_1=\sqrt{\frac{c_2}{2}}\Vec{i},
\quad 
\check{e}_2=\sqrt{\frac{c_1}{2}}\Vec{j},
\quad 
\check{e}_3=\frac{\sqrt{c_1c_2}}{2}\Vec{k}.
\]
The basis $\{\check{e}_1,\check{e}_2,\check{e}_3\}$ satisfies 
the commutation relations
\[
[\check{e}_1,\check{e}_2]=2\check{e}_3,\quad 
[\check{e}_2,\check{e}_3]=c_1\check{e}_1,\quad 
[\check{e}_3,\check{e}_1]=c_2\check{e}_2.
\]
Let us equip an inner product 
$\langle\cdot,\cdot\rangle_{c_1,c_2}$ 
on $\mathfrak{su}(2)$ such that 
$\{\check{e}_1,\check{e}_2,\check{e}_3\}$ is 
orthonormal  with respect to it.
The inner product 
$\langle\cdot,\cdot\rangle_{c_1,c_2}$ 
on $\mathfrak{su}(2)$ 
is related to the standard inner product $\langle\cdot,\cdot\rangle_{1}=-\mathsf{B}/8$ via the 
moment of inertia tensor field $\mathcal{I}$ defined by
\[
\mathcal{I}\Vec{i}=\frac{2}{c_2}\Vec{i},
\quad 
\mathcal{I}\Vec{j}=\frac{2}{c_1}\Vec{j},
\quad 
\mathcal{I}\Vec{k}=\frac{4}{c_1c_2}\Vec{k}.
\]
Denote the left invariant Riemannian metric 
induced from $\langle\cdot,\cdot\rangle_{c_1,c_2}$ 
by $g_{c_1,c_2}$.
The Berger sphere metric is retrieve by the choice 
$c_1=c_2=\frac{c+3}{2}$. The standard metric of $\mathbb{S}^3$ is 
obtained by the choice $c_1=c_2=2$. 
One can see that $\eta_{c_1,c_2}=g_{c_1,c_2}(e_3,\cdot)$ 
is a left invariant contact form on $\mathrm{SU}(2)$. 
Then we get a left invariant magnetic field 
$F_{c_1,c_2}=\mathrm{d}\eta_{c_1,c_2}$. 
The Riemannian metric $g_{c_1,c_2}$ is a compatible 
metric to the contact form $\eta_{c_1,c_2}$ 
(see \textit{e.g.}, \cite{I09,Perrone}). 
The Lorentz force of $F_{c_1,c_2}$ is given by
\[
\varphi \check{e}_1=-\check{e}_2,\quad 
\varphi \check{e}_2=\check{e}_1,\quad 
\varphi \check{e}_3=0. 
\]
The metric $g_{c_1,c_2}$ is naturally reductive if and only if 
at least two of $\{c_1,c_2,2\}$ coincide.
Thus, in case $c_1\not=c_2$, geodesics are not 
necessarily homogeneous. Indeed, 
Marinosci \cite[Theorem 3.1]{Marinosci} proved that 
on $(\mathrm{SU}(2),g_{c_1,c_2})$, 
there exist 
three mutually orthogonal 
homogeneous geodesics through each point. 
In particular, if $c_1,c_2,2$ are all distinct,
there are no other homogeneous geodesics.

Volterra's integration  \cite{Volterra} can be applied 
to the magnetized Euler-Arnold 
equation for the contact magnetic field $F_{c_1,c_2}$ on $(\mathrm{SU}(2),g_{c_1,c_2})$. Even if Volterra carried out 
integration for the magnetized Euler-Arnold equation, 
we need one more integration $\gamma^{-1}(t)\dot{\gamma}
(t)=\mu(t)$ for contact magnetic trajectory $\gamma(t)$. 
The explicit parametrization for contact magnetic 
trajectories in 
$(\mathrm{SU}(2),g_{c_1,c_2},F_{c_1,c_2})$ in terms 
of elliptic functions (\textit{e.g.}, 
Weierstrass $\wp$-function) would be an
interesting project. In case $c_1,c_2,2$ are all distinct, 
then contact magnetic trajectories are 
non-homogeneous, in general, but 
Liouville integrable.
}
\end{remark}

\begin{remark}
{\rm In this article, we concentrate 
on magnetic trajectories 
derived from contact magnetic field 
(standard magnetic field \eqref{eq:homF})
on Berger sphere $\mathscr{M}^{3}(c)$. 
In \cite[Theorem 8.2]{DGJ}, the Liouville integrability of magnetic trajectories 
on the standard unit $3$-sphere $\mathbb{S}^3$ with 
respect to certain magnetic field (generalized 
Demchenko case without twisting) is 
discussed. In particular, elliptic solutions 
to magnetic trajectories are given (see 
\cite[Theorem 9.4]{DGJ}).
}
\end{remark}

\appendix

\section{Other models of Berger sphere}\label{sec:A}
In Appendices we exhibit other models 
of Berger spheres.
\subsection{Berger-Chavel model}
Consider the direct product Lie group
\[
\tilde{L}=\mathrm{SU}(2)\times\mathbb{R}
=\{(A,t)\>|\>A\in\mathrm{SU}(2),\>t\in\mathbb{R}\}
\]
of $G=\mathrm{SU}(2)$ and the abelian group $\tilde{K}=\mathbb{R}$. 
The product Lie group $\tilde{L}$ is 
identified with the closed subgroup
\[
\left\{
\left.
\left(
\begin{array}{ccc}
x_0+\sqrt{-1}x_3 
& -x_2+\sqrt{-1}x_1 & 0\\
x_2+\sqrt{-1}x_1
& x_0-\sqrt{-1}x_3
& 0\\
0 & 0 & e^t
\end{array}
\right)
\>\right|\>
x_0^2+x_1^2+x_2^2+x_3^2=1,\,t\in\mathbb{R}
\right\}
\]
of $\mathrm{GL}_3\mathbb{C}$. Note that 
this linear Lie group is isomorphic to the 
multiplicative group $\mathbf{H}^{\times}
=\mathbf{H}\smallsetminus\{0\}$ of the 
skew field $\mathbf{H}$ of quaternions.

The center $\tilde{Z}$ of $\tilde{L}$ is 
\[
\tilde{Z}=\left\{
\left.
\left(
\begin{array}{cc}
\pm \Vec{1} & 0
\\
0 & e^{t}
\end{array}
\right)
\>\right|
\>t\in\mathbb{R}
\right\}.
\]
\begin{theorem}[\cite{BiRe}]
Any non-trivial discrete central subgroup of $\tilde{L}$ is conjugate to exactly one of 
the following discrete subgroups{\rm:}
\begin{align*}
&\varGamma_1=\left\{
\left(
\begin{array}{cc}
\Vec{1} & 0\\
0 & 1
\end{array}
\right),
\>\> 
\left(
\begin{array}{cc}
-\Vec{1} & 0\\
0 & 1
\end{array}
\right)
\right\},
\quad \>
\varGamma_2=\left\{
\left.
\left(
\begin{array}{cc}
\Vec{1} & 0\\
0 & e^{n}
\end{array}
\right)
\>\right|\>n\in\mathbb{Z}
\right\},
\\
& \varGamma_3=\left\{
\left.
\left(
\begin{array}{cc}
(-1)^n\Vec{1} & 0\\
0 & e^{n}
\end{array}
\right)\>
\right|\>n\in
\mathbb{Z}
\right\},
\quad 
\varGamma_4=
\left\{
\left.
\left(
\begin{array}{cc}
\pm \Vec{1} & 0\\
0 & e^{n}
\end{array}
\right)
\>\right|\>n\in\mathbb{Z}
\right\}.
\end{align*}
\end{theorem}
The Lie algebra $\mathfrak{su}(2)+\mathbb{R}$ of 
$\tilde{L}$ is identified with
\[
\left\{
\left.
x_{1}
\left(
\begin{array}{cc}
\Vec{i} & 0
\\
0 & 0
\end{array}
\right)
+
x_{2}
\left(
\begin{array}{cc}
\Vec{j} & 0
\\
0 & 0
\end{array}
\right)
+
x_{3}
\left(
\begin{array}{cc}
\Vec{k} & 0
\\
0 & 0
\end{array}
\right)
+t
\left(
\begin{array}{cc}
\Vec{0} & 0
\\
0 & 1
\end{array}
\right)
\>\right|\>
x_1,x_2,x_3,t\in\mathbb{R}
\right\}.
\]
The connected Lie groups corresponding to the Lie algebra 
$\mathfrak{su}(2)+\mathbb{R}$ are isomorphic to 
one of the following 
five Lie groups:
\begin{enumerate}
\item $\tilde{L}=\mathrm{SU}(2)\times\mathbb{R}$.
\item $\tilde{L}/\varGamma_1\cong \mathrm{SO}(3)\times\mathbb{R}$.
\item $\tilde{L}/\varGamma_2\cong \mathrm{SU}(2)\times\mathbb{S}^1$.
\item $\tilde{L}/\varGamma_3\cong \mathrm{U}(2)$.
\item $\tilde{L}/\varGamma_4\cong \mathrm{SO}(3)\times\mathbb{S}^1$.
\end{enumerate}

We extend the inner product 
$\langle\cdot,\cdot\rangle_{1}$ of $\mathfrak{su}(2)$ to 
an inner product on 
$\mathfrak{su}(2)+\mathbb{R}$
by the rule
\[
\left\{
\left(
\begin{array}{cc}
\Vec{i} & 0
\\
0 & 0
\end{array}
\right),
\left(
\begin{array}{cc}
\Vec{j} & 0
\\
0 & 0
\end{array}
\right),
\left(
\begin{array}{cc}
\Vec{k} & 0
\\
0 & 0
\end{array}
\right),
\left(
\begin{array}{cc}
\Vec{0} & 0
\\
0 & 1
\end{array}
\right)
\right\}
\]
is orthonormal. 
Then the resulting Riemannian metric 
on $\tilde{L}$ is 
bi-invariant. 
For any $r\in (0,\pi/2]$ we define a 
Lie subalgebra
\[
\mathfrak{h}_{r}=
\mathbb{R}\left\{
\cos r\,
\left(
\begin{array}{cc}
\Vec{k} & 0
\\
0 & 0
\end{array}
\right)+
\sin r
\left(
\begin{array}{cc}
\Vec{0} & 0
\\
0 & 1
\end{array}
\right)
\right\}.
\]
Then the 
corresponding Lie subgroup is 
\[
H_{r}=
\left\{
\left.
\left(
\begin{array}{ccc}
\exp(t\cos r\sqrt{-1}) & 0 & 0\\
0 & \exp(-t\cos r\sqrt{-1}) & 0\\
0 & 0 & \exp(t\sin r)
\end{array}
\right)
\>
\right|
\>t\in\mathbb{R}
\,\right\}.
\]
The orthogonal complement 
$\mathfrak{p}_{r}:=\mathfrak{h}_{r}^{\perp}$ is 
spanned by 
\[
\left\{
\left(
\begin{array}{cc}
\Vec{i} & 0
\\
0 & 0
\end{array}
\right),
\left(
\begin{array}{cc}
\Vec{j} & 0
\\
0 & 0
\end{array}
\right),
-\sin r\,
\left(
\begin{array}{cc}
\Vec{k} & 0
\\
0 & 0
\end{array}
\right)+
\cos r
\left(
\begin{array}{cc}
\Vec{0} & 0
\\
0 & 1
\end{array}
\right)
\right\}.
\]
Let us identify the tangent space  
of $(\mathrm{SU}(2)\times\mathbb{R})/H_{r}$ at the 
origin with $\mathfrak{p}_r$. 
Then $(\mathrm{SU}(2)\times\mathbb{R})/H_{r}$ is a 
normal homogeneous space. The 
normal homogeneous Riemannian space 
$(\mathrm{SU}(2)\times\mathbb{R})/H_{r}$ with 
$r\not=\pi/2$ is the original Berger sphere 
introduced by Berger (see also \cite{Chavel,Sakai81}).
Note that when $r=\pi/2$, 
$(\mathrm{SU}(2)\times\mathbb{R})/H_{\pi/2}=\mathrm{SU}(2)$ is 
a (non-symmetric) normal homogeneous 
representation of the $3$-sphere $\mathbb{S}^3$. 
Sakai \cite{Sakai73} constructed a homothetic Sasakian structure on 
$(\mathrm{SU}(2)\times\mathbb{R})/H_{r}$. 

For higher dimensional generalizations of the Berger sphere, we refer to \cite{OR}.

\subsection{The $\mathrm{U}(2)$-model}
In this section we give a normal homogeneous space representation 
$L/H$ for the Berger sphere with $L=\mathrm{U}(2)
=\{P\in\mathrm{GL}_2\mathbb{C}\>|\>P\,{}^t\!\overline{P}=\Vec{1}\}$.

\subsubsection{}
The Lie algebra $\mathfrak{l}=\mathfrak{u}(2)$ of the 
unitary group $L=\mathrm{U}(2)$ is the direct sum 
of $\mathfrak{su}(2)$ and the center
\[
\mathfrak{z}(\mathfrak{u}(2))
=\mathbb{R}\Vec{z},
\quad 
\Vec{z}=\frac{\sqrt{-1}}{2}\Vec{1}.
\]
Every element $X\in\mathfrak{u}(2)$ is 
expressed as the form
\[
X=\Vec{x}+\frac{\mathrm{tr}\,X}{2}\Vec{1}
=\Vec{x}-\sqrt{-1}\mathrm{tr}\,X\,\Vec{z},
\quad 
\Vec{x}=x_{1}\Vec{i}+x_{2}\Vec{j}+x_{3}\Vec{k}\in\mathfrak{su}(2).
\]
The Killing form $\mathsf{B}_{\mathfrak{u}(2)}$ of 
$\mathfrak{u}(2)$ is given by
\[
\mathsf{B}_{\mathfrak{u}(2)}(X,Y)=
4\mathrm{tr}(XY)-2\mathrm{tr}(X)\mathrm{tr}(Y).
\]
In particular, $\mathsf{B}_{\mathfrak{u}(2)}=0$ on $\mathfrak{z}(\mathfrak{u}(2))$. 
Note that the Killing form $\mathsf{B}$ of $\mathfrak{su}(2)$ is 
the restriction of $\mathsf{B}_{\mathfrak{u}(2)}$ to $\mathfrak{su}(2)$. 
Since $\mathfrak{z}(\mathfrak{u}(2))$ is the center, 
we have $[\mathfrak{z}(\mathfrak{u}(2)),\mathfrak{g}]=\{\Vec{0}\}
\subset\mathfrak{g}$. 
Hence $\mathfrak{u}(2)=\mathfrak{z}(\mathfrak{u}(2))\oplus\mathfrak{g}$ is reductive.
The restriction of 
$\mathsf{B}_{\mathfrak{u}(2)}$ to $\mathfrak{su}(2)$ coincides 
with the Killing form of $\mathfrak{su}(2)$.

We extend the inner product $\langle\cdot,\cdot\rangle_1$ of $\mathfrak{su}(2)$ 
to an inner product $\langle\cdot,\cdot\rangle_{1,\lambda}$ 
on $\mathfrak{u}(2)$ as (see \cite{BV}):
\[
\langle X,Y\rangle_{1,\lambda}=-\frac{1}{2}\mathrm{tr}(\Vec{x}\Vec{y})+\lambda^{2}
\,\mathrm{tr}\,X\,\overline{\mathrm{tr}\,Y},
\quad X,Y\in\mathfrak{u}(2).
\]
Here $\lambda$ is a positive constant. 
The basis $\{\Vec{i},\Vec{j},\Vec{k},\lambda^{-1}\Vec{z}\}$ is 
orthonormal with respect to $\langle \cdot,\cdot\rangle_{1}$.
One can check that the $\mathsf{U}$-tensor \eqref{eq:U} 
of $\mathfrak{u}(2)$ with respect to this inner product vanishes. Thus 
the inner product $\langle \cdot,\cdot\rangle_{1,\lambda}$ induces a bi-invariant Riemannian metric on 
$L=\mathrm{U}(2)$. Note that every bi-invariant Riemannian metric 
on $\mathrm{U}(2)$ is homothetic to the metric induced from $\langle \cdot,\cdot\rangle_{1,\lambda}$ 
(see \cite{BV}).

When $\lambda=1/2$, the inner product 
$\langle X,Y\rangle_{1,1/2}$ coincides with the standard 
inner product:
\[
\langle X,Y\rangle_{1,1/2}=-\frac{1}{2}\mathrm{tr}(XY),
\quad X,Y\in\mathfrak{u}(2).
\]
\begin{remark}{\rm 
The inner product $\langle\cdot,\cdot\rangle_{\alpha}$ used in 
\cite{BV} is written as 
\[
\langle\cdot,\cdot\rangle_{\alpha}=2\langle \cdot,\cdot\rangle_{1,\sqrt{\alpha/2}}.
\]
}
\end{remark}

\subsubsection{}
The 
center $Z(\mathrm{U}(2))$ 
of the unitary group $\mathrm{U}(2)$ is given by
\[
Z(\mathrm{U}(2))
=\exp\mathfrak{z}(\mathfrak{u}(2))
=\left\{\left.
\left(
\begin{array}{cc}
e^{t\sqrt{-1}}/2 & 0
\\
0 & e^{t\sqrt{-1}}/2
\end{array}
\right)
\right
\vert
\>
t\in \mathbb{R}
\right\}.
\]
Obviously $Z(\mathrm{U}(2))$ is isomorphic 
to the circle group $\mathrm{U}(1)=\{z\in\mathbb{C}\>|\>z\bar{z}=1\}$ as 
well as the Lie subgroup $K\subset G$.

\begin{remark}
{\rm 
Via the conjugation, 
$1$-parameter subgroups of $\mathrm{U}(2)$ 
are expressed as
\[
H_{\nu_1,\nu_2}=\left\{\left.
\left(
\begin{array}{cc}
e^{\nu_{1}t\sqrt{-1}} & 0\\
0 & e^{\nu_{2}t\sqrt{-1}}
\end{array}
\right)
\>\right|\>t\in\mathbb{R}
\right\}
\] 
for some $\nu_1$, $\nu_2\in\mathbb{R}$. 
Note that when $\nu_1=\nu_2\not=0$, 
$H_{\nu_1,\nu_1}=Z(\mathrm{U}(2))$.

Moreover, by conjugation and reparametization, we 
can normalize the one-parameter subgroups as
$H_{\nu,1}$ (see \cite[Proposition~3]{Torralbo12}).
}
\end{remark}
One can see that $\mathrm{U}(2)$ is isomorphic 
to the product Lie group $G\times \mathrm{U}(1)$. 
Indeed, let us identify $\mathrm{U}(1)$ with the center $Z(\mathrm{U}(2))$. 
The determinant $\det A$ of $A\in\mathrm{U}(2)$ is a unit 
complex number. So we may represent $\det A$ as 
$\det A=e^{t_A\sqrt{-1}}$.
Then every element $A\in \mathrm{U}(2)$ is decomposed as
\[
A=\left(
A\left(
\begin{array}{cc}
e^{-t_A\sqrt{-1}/2} & 0
\\
0 & e^{-t_A\sqrt{-1}/2}
\end{array}
\right)
\right)
\cdot 
\left(
\begin{array}{cc}
e^{t_A\sqrt{-1}/2} & 0
\\
0 & e^{t_A\sqrt{-1}/2}
\end{array}
\right).
\]
Thus we obtain the Lie group decomposition
\[
\mathrm{U}(2)=G\cdot Z(\mathrm{U}(2)).
\]
This decomposition implies that  
the map 
$\varPsi_0:\mathrm{U}(2)\to G\times K$ by
\[
\varPsi_0(A)=\left(
A
\left(
\begin{array}{cc}
e^{-t_A\sqrt{-1}/2} & 0
\\
0 & e^{-t_A\sqrt{-1}/2}
\end{array}
\right),
\left(
\begin{array}{cc}
e^{-t_A\sqrt{-1}/2} & 0
\\
0 & e^{t_A\sqrt{-1}/2}
\end{array}
\right)
\>
\right)
\]
is a Lie group isomorphism.

Here we describe the action of $L$ on $G$ under 
the identification $L=G\times K$.
Take an element $A\in L$, then split $A$ as
\[
A=A_{G}A_{Z},\quad A_{G}\in G,\>\>
A_{Z}\in Z(L).
\]
The component $A_Z$ is expressred as
\[
A_{Z}=\sqrt{\det A}\,\Vec{1}=\exp(t_{A}\sqrt{-1}/2)\Vec{1}.
\]
Next we set 
\[
A_{K}=\exp(-t_{A}\Vec{k}/2)
=\left(
\begin{array}{cc}
\exp(-t_{A}\sqrt{-1}/2) &0
\\ 0 & \exp(t_{A}\sqrt{-1}/2)
\end{array}
\right)\in K.
\]
Then the action of $A\in L$ on $X\in G$ is defined by
\[
A\cdot X=A_{G}X A_{K}^{-1}.
\]

\begin{example}[Locally polar action]
{\rm Let us consider the 
action of the 
$1$-parameter subgroup $H_{\nu_1,\nu_2}$.
For an element
\[
A=\left(
\begin{array}{cc}
e^{\nu_1 t\sqrt{-1}} &0
\\
0 & e^{\nu_2 t\sqrt{-1}}
\end{array}
\right).
\]
Then $A$ is decomposed as
\[
A_G=
\left(
\begin{array}{cc}
e^{(\nu_1-\nu_2)t\sqrt{-1}/2} &0
\\
0 & e^{-(\nu_1-\nu_2)t\sqrt{-1}/2}
\end{array}
\right),
\quad 
A_K=
\left(
\begin{array}{cc}
e^{-(\nu_1+\nu_2)t\sqrt{-1}/2} &0
\\
0 & e^{(\nu_1+\nu_2)t\sqrt{-1}/2}
\end{array}
\right).
\]
Then for 
\[
X=\left(
\begin{array}{cc}
\alpha & -\bar{\beta}\\
\beta & \bar{\alpha}
\end{array}
\right)\in G,
\]
we have
\[
A_{G}XA_{K}^{-1}=
\left(
\begin{array}{cc}
\alpha e^{\nu_1 t\sqrt{-1}}& -\bar{\beta}e^{-\nu_2 t\sqrt{-1}}\\
\beta e^{\nu_2 t\sqrt{-1}}& \bar{\alpha}e^{-\nu_1 t\sqrt{-1}}
\end{array}
\right).
\]
This action is abbreviated as
\[
\left(e^{\nu_1 t\sqrt{-1}},e^{\nu_2 t\sqrt{-1}}\right)
\cdot (\alpha,\beta)=
\left(\alpha e^{\nu_1 t\sqrt{-1}},\beta e^{\nu_2 t\sqrt{-1}}
\right).
\]
In case $H_{1,1}=Z(L)$,  $A_{G}=\Vec{1}$, 
$A=A_{Z}$ and 
$A_K=\exp(-t\Vec{k})$ for $A\in Z(L)$ and 
we have 
\[
\left(e^{t\sqrt{-1}},e^{t\sqrt{-1}}\right)
\cdot (\alpha,\beta)=
\left(\alpha e^{t\sqrt{-1}},\beta e^{t\sqrt{-1}}
\right).
\]
On the other hand, for the case $H_{1,-1}=K_1$, 
$A=A_G$, $A_Z=\Vec{1}$ and $A_K=A^{-1}$. The action is 
\[
\left(e^{t\sqrt{-1}},e^{-t\sqrt{-1}}\right)
\cdot (\alpha,\beta)=
\left(\alpha e^{t\sqrt{-1}},\beta e^{-t\sqrt{-1}}
\right).
\]
Now let us consider the action of $H_{0,1}$ on $G$. 
In this case $A_G=A_K$ for any $A\in H_{0,1}$. 
Thus $A\cdot X=\mathrm{Ad}(A_K)X$. 
The action is 
\[
\left(1,e^{t\sqrt{-1}}\right)
\cdot (\alpha,\beta)=
\left(\alpha,\beta e^{t\sqrt{-1}}
\right).
\]
Di Scala \cite{Di} proved that the action of $H_{0,1}$ on $G$ is locally polar 
when and only when $c=1$.
}
\end{example}

The product Lie group  
$\tilde{L}=\mathrm{SU}(2)\times\mathbb{R}$ is regarded as the 
universal covering 
of $\mathrm{U}(2)$. Indeed, let us 
define a map $\tilde{\psi}:
\tilde{L}\to \mathrm{U}(2)$ by
\[
\tilde{\psi}\left(\>
\left(
\begin{array}{cc}
\alpha & -\bar{\beta} \\
\beta & \bar{\alpha}
\end{array}
\right),
t
\>
\right)=\left(
\begin{array}{cc}
\alpha & -\bar{\beta} \\
\beta & \bar{\alpha} 
\end{array}
\right)
\left(
\begin{array}{cc}
e^{t\sqrt{-1}/2} & 0 \\
0 & e^{t\sqrt{-1}/2} 
\end{array}
\right).
\]
Then $\tilde{\psi}$ factors through a 
Lie group isomorphism 
$\psi:\mathrm{SU}(2)\times\mathrm{U}(1)\to\mathrm{U}(2)$ and 
$\psi=\varPsi_0^{-1}$ holds. 

\begin{remark}
{\rm Under $\tilde{\psi}$, the isotropy 
subgroup $H_r\subset\mathrm{SU}(2)\times\mathbb{R}$ is 
transformed to 
$H_{\nu_1,\nu_2}$ with
\[
\nu_1=\cos r+\frac{1}{2}\sin r,
\quad 
\nu_2=-\cos r+\frac{1}{2}\sin r.
\]
In particular, when $r=\pi/2$, $H_{\pi/2}$ corresponds to 
$Z(\mathrm{U}(2))$. 
Next if we choose $r=\tan^{-1}2\in (0,\pi/2)$, we have 
$\nu_1=2/\sqrt{5}$ and $\nu_2=0$.
}
\end{remark}
\subsubsection{}

Let us pay attention to 
\[
H=
H_{0,1}=
\left\{\left.
\left(
\begin{array}{cc}
1 & 0
\\
0 & e^{t\sqrt{-1}}
\end{array}
\right)
\>\right|
\>t\in\mathbb{R}
\right\}\cong \mathrm{U}(1)=
\left\{
e^{t\sqrt{-1}}\>|\>t\in\mathbb{R}
\right\}.
\]
Since $L\cap H=\Vec{1}$, 
$L$ is a semi-direct product 
of $G$ and $H\cong\mathrm{U}(1)$ with 
respect to the 
representation
\[
\rho\left(e^{\sqrt{-1}t}\right)A=
\left(
\begin{array}{cc}
1 & 0
\\
0 & e^{t\sqrt{-1}}
\end{array}
\right)
A
\left(
\begin{array}{cc}
1 & 0
\\
0 & e^{-t\sqrt{-1}}
\end{array}
\right).
\]
The multiplication law is explicitly given by
\[
\left(A_1,e^{t_1\sqrt{-1}}\right)
\left(A_2,e^{t_2\sqrt{-1}}\right)
=
\left(A_1\rho_{1}(e^{t_1\sqrt{-1}})A_2,e^{(t_1+t_2)\sqrt{-1}}
\right).
\]
Since
\begin{align*}
&\quad \left(
\>
\left(
\begin{array}{cc}
\alpha_1 &-\bar{\beta}_1
\\
\beta_1 &\bar{\alpha}_1
\end{array}
\right),
\>
e^{t_1\sqrt{1}}
\right)
\left(
\>
\left(
\begin{array}{cc}
\alpha_2 &-\bar{\beta}_2
\\
\beta_2 &\bar{\alpha}_2
\end{array}
\right),
\>
e^{t_2\sqrt{1}}
\right)
\\
&=\left(
\left(
\begin{array}{cc}
\alpha_1 &-\bar{\beta}_1
\\
\beta_1 &\bar{\alpha}_1
\end{array}
\right)
\left(
\begin{array}{cc}
\alpha_2 &-\bar{\beta}_2e^{-t_1\sqrt{-1}}
\\
\beta_2e^{t_1\sqrt{-1}} &\bar{\alpha}_2
\end{array}
\right),
e^{(t_1+t_2)\sqrt{-1}}
\right),
\end{align*}
the multiplication law is abbreviated as 
\[
\left(\alpha_1,\beta_1,e^{t_1\sqrt{-1}}\right)
\left(\alpha_2,\beta_2,e^{t_2\sqrt{-1}}\right)
=
\left(\alpha_1\alpha_2-\bar{\beta}_1\beta_2e^{t_1\sqrt{-1}},
\alpha_2\beta_1+\bar{\alpha}_1\beta_2e^{t_1\sqrt{-1}},
e^{(t_1+t_2)\sqrt{-1}}\,\right).
\]

\subsubsection{}
Let us consider the semi-direct product 
$L=\mathrm{SU}(2)\ltimes_{\rho}H\cong \mathrm{U}(2)$. 
The Lie algebra $\mathfrak{h}$ of $H$ is 
spanned by
\[
\Vec{h}=\Vec{z}-\frac{1}{2}\Vec{k}=
\left(
\begin{array}{cc}
0& 0
\\
0 & \sqrt{-1}
\end{array}
\right).
\]
Then $\{\Vec{i},\Vec{j},\Vec{k},\Vec{h}\}$ is a 
basis of $\mathfrak{l}=\mathfrak{u}(2)$ and 
$H=\exp\mathfrak{h}$. The vector 
$\Vec{h}$ satisfies the commutation relations
\[
[\Vec{h},\Vec{i}]=-\Vec{j},
\quad 
[\Vec{h},\Vec{j}]=\Vec{i},
\quad 
[\Vec{h},\Vec{k}]=\Vec{0}.
\]
The orthogonal complement 
$\mathfrak{p}:=\mathfrak{h}^{\perp}$ of 
$\mathfrak{h}$ is spanned by the 
orthonormal basis
\[
\left\{
\Vec{i},\Vec{j},
\frac{2\lambda}{\sqrt{1+4\lambda^2}}
\left(
\Vec{k}+\frac{\Vec{z}}{2\lambda^2}
\right)
\right\}.
\]
Since
\[
\Vec{k}+\frac{\Vec{z}}{2\lambda^2}
=\frac{1+4\lambda^2}{4\lambda^2}\Vec{k}
+\frac{1}{2\lambda^2}\Vec{h},
\]
the above orthonormal basis is rewritten as
\[
\left\{
\Vec{i},\Vec{j},
\frac{\sqrt{1+4\lambda^2}}{2\lambda}
\left(
\Vec{k}+\frac{2}{1+4\lambda^2}\Vec{h}
\right)
\right\}.
\]
The resulting homogeneous Riemannian space 
$L/H$ is normal homogeneous.

\subsubsection{}
Now let us perform the homothetical change of the inner 
product as
\[
\langle\cdot,\cdot\rangle_{c}=\frac{4}{c+3}\langle\cdot,\cdot\rangle_{1,\lambda}
\]
on $\mathfrak{u}(2)$ for $c>-3$. 
Then we have the following orthonormal basis
\[
\bar{e}_1=\frac{\sqrt{c+3}}{2}\Vec{i},
\>\>
\bar{e}_2=\frac{\sqrt{c+3}}{2}\Vec{j},
\>\>
\bar{e}_3=
\frac{\sqrt{c+3}}{2}
\frac{\sqrt{1+4\lambda^2}}{2\lambda}
\left(
\Vec{k}+\frac{2}{1+4\lambda^2}\Vec{h}
\right)
\]
of $\mathfrak{n}$. 

Hereafter we assume that $c>1$ and choose
\[
\lambda=\frac{1}{\sqrt{c-1}}>0.
\]
Then we have 
\[
\bar{e}_1=\frac{\sqrt{c+3}}{2}\Vec{i},
\>\>
\bar{e}_2=\frac{\sqrt{c+3}}{2}\Vec{j},
\>\>
\bar{e}_3=
\frac{c+3}{4}
\left(
\Vec{k}+\frac{2(c-1)}{c+3}\Vec{h}
\right).
\]
\begin{remark}
{\rm The case $\lambda=1/2$ corresponds to $c=5$. Thus 
the Berger sphere $\mathscr{M}^3(5)$ 
is represented as a normal homogeneous space 
$\mathrm{U}(2)/\mathrm{U}(1)$ with respect to the 
Riemannian metric induced from the standard 
bi-invariant Riemannian metric on $\mathrm{U}(2)$. 
On $\mathscr{M}^3(5)$ we have
\[
\bar{e}_1=\sqrt{2}\Vec{i},
\quad 
\bar{e}_2=\sqrt{2}\Vec{j},
\quad 
\Vec{e}_3=2(\Vec{k}+\Vec{h}).
\]
}
\end{remark}
The linear isomorphism 
\[
e_1\longmapsto \bar{e}_1,
\quad 
e_2\longmapsto \bar{e}_2,
\quad 
e_3\longmapsto \bar{e}_3
\]
is a linear isometry from 
$\mathfrak{m}$ to $\mathfrak{n}$.
Thus the normal homogeneous space 
$L/H$ is isometric to the 
Berger sphere $\mathscr{M}^{3}(c)$ with $c>1$.
Note that the reductive decomposition 
$\mathfrak{u}(2)=\mathfrak{h}\oplus\mathfrak{n}$ 
coincides with the naturally reductive 
decomposition given in
\cite[Theorem 3.1]{GO} with $t=\varepsilon=4/(c+3)$.

\begin{remark}\rm
The diameter of $\mathscr{M}^{3}(c)$ is computed 
in \cite[Corollary 3.18]{Engel}.
\begin{itemize}
\item If $c\leq 1$, the diameter is $\displaystyle 
\frac{2\pi}{\sqrt{c+3}}$.
\item If $1<c<5$, then 
the diameter is $\displaystyle 
\frac{4\pi}{c+3}$.
\item If $c\geq 5$, 
the diameter is $\displaystyle 
\frac{\pi}{\sqrt{c-1}}$.
\end{itemize}
In the previous remark we pointed out that $c=5$ is, in some sense, 
special. This observation is consistent with the above 
table of the diameter.  Note that 
$\mathscr{M}^{3}(5)$ is isomorphic to the geodesic 
sphere of radius $\pi/4$ in the complex 
projective plane $\mathbb{C}P^2(4)$ of constant holomorphic sectional 
curavture $4$ as a Sasakian manifold. The diameter 
was computed independently by Podobryaev in \cite{Pod}.
\end{remark}
\begin{remark}{\rm As is well known 
the standard almost complex structure 
$J$ on a Riemannian product $M^3\times\mathbb{R}$ of a Sasakian $3$-manifold 
$M^3$ and the real line $\mathbb{R}$ is integrable because the integrability 
of $J$ is 
the normality of almost contact Riemannian structure. 
The resulting Hermitian surface 
$(M^3\times\mathbb{R},J)$ is a Vaisman surface. 
Thus the Riemannian product 
$\mathscr{M}^3(c)\times\mathbb{R}$ 
of a Berger sphere and the real line is a 
Vaisman surface. On the other hand, 
Sasaki classified left invariant 
complex structures on $\mathrm{U}(2)$ \cite{Sasaki}. 
For more information on left invariant LCK structures on $\mathrm{U}(2)$, 
we refer to \cite{CoHa}.  
}
\end{remark}

\section{Homogeneous projections}
\subsection{}
In \cite{AvS}, Arvanitoyeorgos and Souris 
investigated magnetic trajectories in certain classes of homogeneous Riemannian spaces. In particular they obtained the 
following result \cite[Corollary 1.2]{AvS}:

\begin{theorem}\label{thm:AvS}
Let $L/H$ be a homogeneous Riemannian space with $L$ compact, 
$H$ having non-discrete center and reductive decomposition 
$\mathfrak{l}=\mathfrak{h}\oplus\mathfrak{p}$. Consider a homogeneous 
fibration $\pi:L/H\to L/N$ with fiber $N/H$, where $N$ is a closed
Lie subgroup of $L$, so that 
$T_{o}L/H=T_{\pi(o)}G/N\oplus T_{o}N/H$. Endow 
$L/H$ with the $L$-invariant metric $g$ induced from 
the inner product
\begin{equation}\label{eq:AS-metric}
\langle\cdot,\cdot\rangle=
\langle\cdot,\cdot\rangle_{\mathsf B}\vert_{T_{\pi(o)}L/N}+\beta
\langle\cdot,\cdot\rangle_{\mathsf B}\vert_{T_{o}N/H},
\quad 
\langle\cdot,\cdot\rangle_{\mathsf B}=(-\alpha)\mathsf{B}
\end{equation}
where $\mathsf{B}$ is the Killing form of $\mathfrak{l}$ and 
$\alpha$ and $\beta$ are positive constants. Let $\gamma(s)$ be the magnetic trajectory through the origin with charge $q$, under the magnetic field $F^{\zeta}$ defined by \eqref{eq:homF}. Then $\gamma(s)$ is given by 
the equation
\begin{equation}
\gamma(t)=
\exp\left\{
t\left(
\hat{X}_a+\beta \hat{X}_b+q\zeta
\right)
\right\}\,
\exp\left\{
t(1-\beta)\left(
\hat{X}_b+\frac{q}{\beta}\zeta
\right)
\right\}\cdot o.
\end{equation}
Here $\hat{X}_a$ and $\hat{X}_b$ are projections of $\dot{\gamma}(0)$ 
on $\mathfrak{p}_a=T_{\pi(o)}L/N$ and 
$\mathfrak{p}_b=T_{o}N/H$, respectively.
\end{theorem}
Let us apply this result to the Boothby-Wang fibration 
(Hopf fibration)
$\mathscr{M}^3(1)\to\mathbb{S}^2(4)$ (see \cite[Example 5.2, \S 6.1]{AvS0}).

\subsection{}
We regard the unit $3$-sphere $\mathscr{M}^3(1)=(\mathbb{S}^3,g_1)$ 
as the naturally reductive homogeneous space 
$\mathscr{M}^3(1)=(\mathrm{SU}(2)\times\mathrm{U}(1))/\Delta \mathrm{U}(1)$. On $\mathscr{M}^3(1)$, we equip the Riemannian metric 
\[
\tilde{g}=\frac{4}{c+3}g_1,\quad c>-3.
\]
The resulting Riemannian $3$-manifold 
is of constant curvature $(c+3)/4$ and hence isometric to the 
$3$-sphere $\mathbb{S}^3(2/\sqrt{c+3})$ of radius $2/\sqrt{c+3}$. Choose 
\[
L=G\times K=\mathrm{SU}(2)\times\mathrm{U}(1),
\quad 
H=\Delta K=\Delta \mathrm{U}(1),
\quad 
N=K\times K=\mathrm{U}(1)\times\mathrm{U}(1),
\]
then, the  bottom space $L/N$ is 
\[
L/N=(G\times K)/(K\times K)\cong G/K.
\]
The reductive decomposition $\mathfrak{g}\oplus\mathfrak{k}=\Delta\mathfrak{k}
\oplus\widetilde{\mathfrak{p}}$ is given by 
\[
\widetilde{\mathfrak{p}}
=\left\{\left.
\left(
V+W,0
\right)
\>
\right|
\>
V\in\mathfrak{m},\,W\in\mathfrak{k}
\right\},
\]
where 
\[
\mathfrak{m}=\mathbb{R}\Vec{i}
\oplus\mathbb{R}\Vec{j}
\]
as before. Next, the fiber $N/H$ is
\[
N/H=(\mathrm{U}(1)\times\mathrm{U}(1))/\Delta\mathrm{U}(1)\cong \mathrm{U}(1).
\]
The tangent space $\tilde{\mathfrak{p}}$
has the splitting $\tilde{\mathfrak{p}}=\tilde{\mathfrak{p}}_a\oplus\tilde{\mathfrak{p}}_b$ with
\begin{align*}
& \tilde{\mathfrak{p}}_a=T_{\pi(o)}L/N
=\left\{\left.
\left(
V,0
\right)
\>
\right|
\>
V\in\mathfrak{m}
\right\}\cong\mathfrak{m},
\\
& \tilde{\mathfrak{p}}_b=T_{o}N/H
=\left\{\left.
\left(
W,0
\right)
\>
\right|
\>
W\in\mathfrak{k}
\right\}\cong\mathfrak{k}.
\end{align*}
Let us choose $\beta=4/(c+3)$ and 
consider the homogeneous Riemannian metric induced 
from the inner product \eqref{eq:AS-metric}. 
The resulting metric coincides with the Riemannian metric $g$ of the 
Berger sphere $\mathscr{M}^3(c)$. 

The center $\mathfrak{z}(\mathfrak{h})$ is given by
\[
\mathfrak{z}(\mathfrak{h})=\mathfrak{z}(\Delta\mathfrak{u}(1))=\Delta\mathfrak{u}(1)\cong
\mathfrak{u}(1).
\]
Then the contact magnetic field 
$F=\mathrm{d}\eta$ is identified with 
the magnetic field $F^{\zeta}$ defined by
\[
\zeta=
\left(\frac{1}{2}\xi,
\frac{1}{2}\xi
\right),
\quad 
\xi=\frac{c+3}{4}\xi_1.
\]

Then, for 
$\hat{X}=\dot{\gamma}(0)$, we have the decomposition
$\hat{X}=\hat{X}_a+\hat{X}_b$
\[
\hat{X}=(X,0),\quad 
\hat{X}_a=(X_{\mathfrak m},0),
\quad 
\hat{X}_b=(X_{\mathfrak k},0),
\]
where $X=X_{\mathfrak{k}}+X_{\mathfrak m}
\in\mathfrak{k}\oplus\mathfrak{m}$.
Hence we get
\begin{align*}
\hat{X}_a+\beta \hat{X}_b+q\zeta
=&\left(
X_{\mathfrak m}+\frac{4}{c+3}X_{\mathfrak k}+\frac{q}{2}\xi,
\frac{q}{2}\xi
\right),
\\
(1-\beta)\left(
\hat{X}_{b}+\frac{q}{\beta}\zeta
\right)=&\frac{c-1}{c+3}
\left(X_{\mathfrak k}+\frac{(c+3)q}{8}\xi,\frac{(c+3)q}{8}\xi
\right).
\end{align*}
From these formulas, we get
\[
\exp_{G\times K}
\left\{
t(\hat{X}_a+\beta \hat{X}_b+q\zeta)
\right\}
=
\left(
\exp_{G}\left\{t\left(
X_{\mathfrak m}+\frac{4}{c+3}X_{\mathfrak k}+\frac{q}{2}\xi
\right)
\right\},
\exp_{K}\left\{t
\left(
\frac{q}{2}\xi
\right)
\right\}
\right),
\]
\begin{align*}
&\quad \exp_{G\times K}
\left\{
t(1-\beta)\left(
\hat{X}_{b}+\frac{q}{\beta}\zeta
\right)
\right\}
\\
&=
\left(
\exp_{G}
\left\{
\frac{(c-1)t}{c+3}
\left(X_{\mathfrak k}+\frac{(c+3)q}{8}\xi\right)
\right\},
\exp_{K}
\left(
\frac{(c-1)qt}{8}
\xi
\right)
\right).
\end{align*}
Hence
\begin{align*}
& 
\exp_{G\times K}\left\{
t\left(
\hat{X}_a+\beta \hat{X}_b+q\zeta
\right)
\right\}\,
\exp_{G\times K}\left\{
s(1-\beta)\left(
\hat{X}_b+\frac{q}{\beta}\zeta
\right)
\right\}
\\
&=\left(
\begin{array}{c}
\exp_{G}
\left\{
\frac{(c-1)t}{c+3}
\left(X_{\mathfrak k}+\frac{(c+3)q}{8}\xi\right)
\right\}\exp_{K}
\left\{
\frac{(c-1)t}{c+3}
\left(X_{\mathfrak k}+\frac{(c+3)q}{8}\xi\right)
\right\}
\\
\exp_{K}
\left(
\frac{(c+3)qt}{8}
\xi
\right)
\end{array}
\right).
\end{align*}
Henceforth we obtain the formula for the 
contact magnetic trajectory
\begin{equation}
\gamma(t)=
\exp_{G}
\left\{t\left(
X_{\mathfrak m}+\frac{4}{c+3}X_{\mathfrak k}+\frac{q}{2}\xi
\right)
\right\}
\exp_{K}
\left(
\frac{(c-1)t}{c+3}
X_{\mathfrak k}
\right).
\end{equation}
In particular, when $c=1$, we get
\[
\gamma(t)=
\exp_{G}
\left\{t\left(
X+\frac{q}{2}\xi
\right)
\right\}.
\]
This formula coincides with the one 
obtained by Bolsinov and Jovanovi{\'c} \cite[Remark 1]{BJ}. 

We can deduce the explicit parametrization 
for contact magnetic trajectories in the Berger $3$-sphere 
$\mathscr{M}^{3}(c)$ by virtue of Theorem 
\ref{thm:AvS} due to Arvanitoyeorgos and Souris. 
However our method (the use of magnetized Euler-Arnold equation) has some 
advantages. First, our method is elementary and direct. 
Next, our method can be applied to the case $L$ is 
\emph{non-compact}. Indeed, in our previous works \cite{IM24}, 
we proved the homogeneity of contact magnetic trajectories 
of $\mathrm{SL}_2\mathbb{R}=(\mathrm{SL}_2\mathbb{R}\times\mathrm{SO}(2))/(\mathrm{SO}(2)\times\mathrm{SO}(2))$. 
Moreover, our formula \eqref{eq:hom-cont-mag} implies that 
every contact magnetic trajectory is obtained by right translations 
of a homogeneous geodesic under the charged Reeb flow.

\section{Hamiltonian formulation of magnetic trajectories}\label{sec:C}

\subsection{The canonical two-form}
Let $M$ be an $m$-manifold. Denote by $T^{*}M$ its cotangent bundle.
Take a local coordinate system $(x^1,x^2,\dots,x^m)$ of $M$, then
it induces a fiber coordinates $(p_1,p_2,\dots,p_m)$. 
Moreover, $\vartheta=\sum_{j=1}^{m}p_{j}\,dx^j$ is a globally defined one-form
on $T^{*}M$ and called the
\emph{canonical one-form} (also called the \emph{Liouville form}).
The two-form $\varPsi=-d\vartheta=dx^{j}\wedge dp_{j}$ is a
symplectic form on $T^{*}M$.
The symplectic form $\varPsi$ is called the \emph{canonical two-form} on $T^{*}M$.
For a prescribed smooth function $H$, the 
\emph{Hamiltonian vector field} 
$X_H$ with \emph{Hamiltonian} $H$ is 
defined by
\[
\mathrm{d}H(Y)=\Psi(X_{H},Y),\quad Y\in \varGamma(T(T^{*}M)).
\]
The Hamiltonian vector field $X_{H}$
is locally expressed as
\[
X_{H}=\sum_{i=1}^{m}\frac{\partial H}{\partial p_i}\frac{\partial}{\partial x^i}
-\sum_{i=1}^{m}\frac{\partial H}{\partial x^i}\frac{\partial}{\partial p_i}.
\]
The dynamical system $(T^{*}M,\Psi,H)$ is called a 
\emph{Hamiltonian 
system} with \emph{configuration space} $M$ and 
Hamiltonian $H$. The cotangent bundle $T^{*}M$ is 
referred as to the \emph{phase space}.

Take a curve 
$\bar{\gamma}(t)=(x^{1}(t),x^{2}(t),\dots,
x^{m}(t);p_{1}(t),p_{2}(t),\dots,p_{m}(t))$ 
in the phase space $T^{*}M$.
Then $\bar{\gamma}(t)$ is an integral curve of $X_H$ if and only if
it satisfies the \emph{Hamilton equation}:
\[
\frac{\mathrm{d}x^i}{\mathrm{d}t}=\frac{\partial H}{\partial p_i},
\quad
\frac{\mathrm{d}p_i}{\mathrm{d}t}=-\frac{\partial H}{\partial x^i},
\quad 1\leq i\leq m.
\]
\subsection{Geodesic flows}
Let us consider a Riemannian $m$-manifold $(M,g)$, then its
cotangent bundle $T^{*}M$ is identified with the
tangent bundle $TM$ via the metric $g$. The so-called musical isomorphism
\[
\flat:TM\to T^{*}M;\quad \flat v=g_{p}(v,\cdot),\quad v\in T_{p}M
\]
is a vector bundle isomorphism. We denote by $\pi_{TM}$ the
natural projection of $TM$ onto $M$.
Take a local coordinate system $(x^1,x^2,\dots,x^m)$
with fiber coordinates $(u^1,u^2,\dots,u^m)$. Then
the pull-backed one-form $\flat^{*}\vartheta$ is
expressed as $\flat^{*}\vartheta=\sum\limits_{i,j=1}^{m}g_{ij}u^{j}dx^{i}$.
Then the pull-backed two-form $\varPhi:=\flat^{*}\varPsi$ is
computed as
\[
\flat^{*}\varPsi=\sum_{i,j=1}^{m}
g_{ij}\mathrm{d}x^{i}\wedge \mathrm{d}u^{j}
+\sum_{i,j,k=1}^{m}\frac{\partial g_{ji}}{\partial x^k}\,u^{j}\,
\mathrm{d}x^{i}\wedge \mathrm{d}x^{k}.
\]
The pull-backed two-form $\varPhi$ is a symplectic form on $TM$.
Let us consider the \emph{kinetic energy} on $TM$:
\[
E(p;v)=\frac{1}{2}g_{p}(v,v),\quad v\in T_{p}M.
\]
Then we get a Hamiltonian system $(TM,\varPhi, E)$.
The Hamiltonian vector field on $TM$ with
Hamiltonian $E$ is called the
\emph{geodesic spray} and has local expression:
\[
X_{E}=\sum_{i=1}^{m}u^{i}\frac{\partial}{\partial x^i}
-\sum_{i,j,k=1}^{m}\varGamma_{ij}^{\>k}u^{i}u^{j}\frac{\partial}{\partial u^k}.
\]
The integral curves of $\xi_E$ are solutions of the system
\[
\frac{\mathrm{d}x^k}{\mathrm{d}t}=u^{k},\quad
\frac{\mathrm{d}u^k}{\mathrm{d}t}=-\sum_{i,j=1}^{m}\varGamma_{ij}^{\>k}u^{i}u^{j}.
\]
One can see that for every
trajectory $\bar{\gamma}(t)$ of the
Hamiltonian system $(TM,\varPhi,E)$,
the projected curve
$\gamma(t)=\pi(\gamma(t))$ in $M$ is a geodesic.

\subsection{Magnetic trajectories}
Let us consider a magnetic field $F$ on a Riemanian manifold $(M,g)$.
Express $F$ and its Lorentz force $\varphi$ as
\[
F=\sum_{i<j}F_{ij}
\mathrm{d}x^{i}\wedge 
\mathrm{d}x^{j}=
\frac{1}{2}\sum_{i,j=1}^{m}F_{ij} 
\mathrm{d}x^{i}\wedge \mathrm{d}x^{j},
\quad 
\varphi \frac{\partial}{\partial x^i}=\sum_{k=1}^{m}\varphi_{i}^{\>k}\frac{\partial}{\partial x^k},
\quad 
\varphi_{i}^{\>l}=-\sum_{j=1}^{m}g^{lj}F_{ji}.
\]
Let us magnetize the symplectic form $\varPhi$ of $TM$  as
\[
\varPhi_F=\varPhi+q\pi_{TM}^{*}F
\]
for some constant $q$. 
The deformed two-form $\varPhi_F$ is still
a symplectic form (and hence $\varPhi_F$ itself is a magnetic field). Let us consider the Hamiltonian system
$(TM,\varPhi_F,E)$. The Hamiltonian vector field
is given by
\[
X_E^F=X_{E}+q\mathrm{v}\{\varphi(u)\},
\]
where $\mathrm{v}\{\varphi(u)\}$ is a vertical vector field
on $TM$ globally defined by
\[
\mathrm{v}\{\varphi(u)\}=\sum_{i,k=1}^{m}\varphi_{i}^{\>k}u^{i}\frac{\partial}{\partial u^k}.
\]

The Hamilton equation is
\[
\frac{\mathrm{d}x^k}{\mathrm{d}t}=u^{k},\quad
\frac{\mathrm{d}u^k}{\mathrm{d}t}=-\sum_{i,j=1}^{m}\varGamma_{ij}^{\>k}u^iu^j+q\sum_{i=1}^{m}\varphi_{i}^{\>k}u^i.
\]
This is a second order ODE
\[
\frac{\mathrm{d}^2x^k}{\mathrm{d}t^2}
+\sum_{i,j=1}^{m}\varGamma^{\>k}_{ij}
\frac{\mathrm{d}x^i}{dt}
\frac{\mathrm{d}x^j}{\mathrm{d}t}
=q\,\sum_{i=1}^{m}\varphi_{i}^{k}\frac{\mathrm{d}x^i}{\mathrm{d}t}.
\]
This system has coordinate-free expression 
\[
\nabla_{\dot{\gamma}}\dot{\gamma}=q\varphi\dot{\gamma}.
\]
This in nothing but the Lorentz equation \eqref{eq:Lorentz}.
\begin{proposition}
Let $(M,g)$ be a Riemannian manifold with a magnetic field
$F$. Then the magnetic trajectory equation is
the Hamilton equation of the Hamiltonian system
$(TM,\varPhi_F,E)$.
\end{proposition}

\subsection{Lagrangian formalism of magnetic trajectories}
Let $(M,g)$ be a Riemannian manifold 
equipped with an exact magnetic field 
$F=\mathrm{d}A$, where $A$ is the magnetic potential $1$-form. Then $\gamma(t)$ is a magnetic trajectory 
of charge $q$ under the influence of $F$ if and only 
if it is a critical point of the 
\emph{Landau-Hall 
functional}
\[
\mathrm{LH}(\gamma)=\int \frac{1}{2}
g(\dot{\gamma}(t),\dot{\gamma}(t))\,\mathrm{d}t+q 
\int A(\dot{\gamma}(t))\,\mathrm{d}t.
\]
Thus the magnetic trajectory equation is 
understood as a Lagrangian system 
$(TM,\mathrm{LH})$ with Lagrangian $\mathrm{LH}$.

\subsection{Non-holonomic Lagrangian systems 
with gyroscopic forces}

Let $(M,g)$ be a Riemannian $m$-manifold with an exact magnetic field $F=\mathrm{d}A$. Consider a 
\emph{non-holonomic} Lagrangian system  
$(TM,L_1,D)$, where the constraints define a
non-integrable distribution $D\subset TM$. 
Assume that the constraints are homogeneous and do not 
depend on time. 
In \cite{DGJ}, the authors consider 
Lagrangians $L_1$, along 
with the difference of the kinetic energy 
and potential energy, contains an additional 
term, which is linear in velocities:
\[
L_{1}(x^1,x^2,\dots,x^m;u^1,u^2,\dots,u^m)=
\frac{1}{2}\sum_{i,j=1}^{m}g_{ij}u^iu^j
+\sum_{k=1}^{m}A_{k}u^{k}-V(x^1,x^2,\dots,x^m).
\]
Here the smooth function  
\[
(x^1,x^2,\dots,x^m;u^{1},u^{2},\dots,u^{m})
\longmapsto 
\sum_{k=1}^{m}A_{k}(x^1,x^2,\dots,x^{m};u^{1},u^{2},\dots,u^{m})u^k
\]
is interpreted as a one-form 
$A=\sum_{i=1}^{m}A_{i}\mathrm{d}x^i$ on $M$. 
The one-form $A$ is understood as a magnetic potential. 
Obviously, the Lagrangian $L_{1}$ has the 
form $L_{1}=\mathrm{LH}_{q=1}-V$, where 
$\mathrm{LH}_{q=1}$ is the Landau-Hall functional with charge $q=1$.

Take a path $\gamma(t)=(x^1(t),x^2(t),\dots,x^m(t))$ 
in the configuration space $(M,g)$. 
Then $\gamma(t)$ is said to be \emph{admissible} 
if its velocity $\dot{\gamma}(t)$ belongs to 
the distribution $D$.
The first variation formula for the action integral
\[
\mathcal{L}_1(\gamma)=\int_{0}^{1}
L_{1}(\gamma(t),\dot{\gamma}(t))\,\mathrm{d}t
\]
with respect to admissible 
variation $\{\gamma_{\varepsilon}(t)
=(x^{1}_{\varepsilon}(t),x^{2}_{\varepsilon}(t),\dots, x^{m}_{\varepsilon}(t))\}$
is given by
\[
\frac{\mathrm{d}}{\mathrm{d}\varepsilon}
\biggr\vert_{\varepsilon=0}
\mathcal{L}_1(\gamma_{\varepsilon})
=\int_{0}^{1}
\sum_{i=1}^{m}
\left\{
\frac{\partial L_1}{\partial x^i}
–\frac{\mathrm{d}}{\mathrm{d}t}
\left(
\frac{\partial L_1}{\partial \dot{x}^i}
\right)
\right\}v^{i}
\,\mathrm{d}t,
\quad v^{i}=
\frac{\mathrm{d}}{\mathrm{d}\varepsilon}
\biggr\vert_{\varepsilon=0}x^{i}_{\varepsilon}(t).
\]
Here the variation $\gamma(t)\longmapsto \gamma_{\varepsilon}(t)$ through an admissible curve $\gamma_{0}(t)=\gamma(t)$ is 
said to be \emph{admissible} if it is an admissible curve for any 
$\varepsilon$. Thus its variational vector field satisfies
\[
\sum_{i=1}^{m}v^{i}(t)\frac{\partial}{\partial x^i}
\biggr\vert_{\gamma(t)}\in D_{\gamma(t)}
\]
for any $t$.
Set $L=E-V$, where $E$ is the kinetic energy. 
According to the notation of \cite{DGJ}, the dynamical system $(TM,L_1,D)$ together with $A$ is 
denoted by $(M,L,F,D)$ and called 
a \emph{natural mechanical non-holonomic 
system with a magnetic force}.

\begin{example}[Chaplygin systems]
{\rm 
Let $M$ be a principal circle bundle over a Riemannian manifold $(B,\bar{g})$. Let us assume that $M$ has 
a Riemannian metric $g$ so that the projection 
$\pi:M\to B$ is a Riemannian submersion and 
the horizontal distribution $\mathcal{H}$ coincides with 
Ehresmann connection of the circle bundle. In additin, we 
equip a horizontal magnetic field $F$ on $M$. 
Then the dynamical system $(M,L,F,\mathcal{H})$ is called a 
\emph{gyroscopic Chaplygin system} (see \cite[Definition 1.2]{DGJ}). 

For instance, the Berger sphere $\mathscr{M}^3(c)$ is a
principal circle bundle over $\mathbb{S}^2(1/\sqrt{c+3})$. 
The Boothby-Wang fibering $\mathscr{M}^3(c)\to\mathbb{S}^2(1/\sqrt{c+3})$ is a Riemannian submersion. The horizontal 
distribution is nothing but the 
contact structure and coincides with Ehresmann connection. 
The contact magnetic field $F=\mathrm{d}\eta$ is horizontal. Thus 
$(\mathscr{M}^3(c),L,\mathcal{D},\mathrm{d}\eta)$ is an example 
of gyroscopic Chaplygin system.
}    
\end{example}

\bigskip

\noindent
{\bf Acknowledgments.} 
The first named author was partially supported by JSPS KAKENHI JP19K03461, JP23K03081.
The second named author was partially supported by FCSU grant of the Alexandru Ioan Cuza
 University, as well as by the Romanian Ministry of Research, 
Innovation and Digitization, within Program 1 – Development of the national 
RD system, Subprogram 1.2 – Institutional Performance – RDI excellence funding projects, 
Contract no.11PFE/30.12.2021.  

The authors would like to express their sincere gratitude to the referee for her/his careful reading of the manuscript and suggestions concerning on connections between our results and rigid body dynamics.




\end{document}